\documentclass[a4paper]{article}
\usepackage{amsmath,amssymb,amsthm}
\usepackage[T1]{fontenc}
\usepackage[latin1]{inputenc}
\usepackage{hyperref}
\usepackage{comment}

\newcommand{\U}{\mathcal{U}}

\newcommand{\M}{\mathcal{M}}
\renewcommand{\L}{\mathcal{L}}

\renewcommand{\d}{\mathfrak{d}_{\textsc{J}_1}}

\newcommand{\N}{\mathbb{N}}
\newcommand{\E}{\mathcal{E}}
\renewcommand{\SS}{\mathbb{S}}
\newcommand{\C}{\mathcal{C}}
\newcommand{\D}{\mathcal{D}}

\newcommand{\R}{\mathbb{R}}
\newcommand{\s}{\mathcal{s}}
\newcommand{\I}{1{\hskip -2.5 pt}\hbox{I}}
\newcommand{\II}{\mathbf{I}_\phi}

\newcommand{\X}{\mathcal{S}}

\theoremstyle{plain}
\newtheorem{Thm}{Theorem}[section]
\newtheorem{Lem}[Thm]{Lemma}
\newtheorem{Prop}[Thm]{Proposition}
\newtheorem*{Prop*}{Proposition}
\newtheorem{Cor}[Thm]{Corollary}

\theoremstyle{definition}
\newtheorem{Def}[Thm]{Definition}
\newtheorem{Eg}[Thm]{Example}
\newtheorem{assumption}[Thm]{Assumption}
\theoremstyle{remark}
\newtheorem{Rem}[Thm]{Remark}
\def \follmer {F\"{o}llmer}
\def \ito {It\^{o} }

\def \cadlag {c\`adl\`ag }

\def \caglad {c\`agl\`ad }
\def \closed {generic} 

\begin{document}

\title{Causal functional calculus}



\author{Henry CHIU \footnote{Dept of Mathematics, Imperial College London.
{\tt    h.chiu16@imperial.ac.uk}}
  \and 
  Rama CONT\footnote{Mathematical Institute, University of Oxford.
{\tt  Rama.Cont@maths.ox.ac.uk} 
  }}

\date{July 2022}

\maketitle

\begin{abstract}
We construct a new topology on the space of stopped paths and introduce a calculus for causal functionals on generic domains of this space. We propose a generic approach to pathwise integration without any assumption on the variation index of a path and obtain functional change of variable formulas which extend the results of \follmer\ (1981) and Cont \& Fourni\'e (2010) to a larger class of functionals, including \follmer's pathwise integrals. We show that a class of smooth functionals possess a pathwise analogue of the martingale property. For paths that possess finite quadratic variation, our approach extends F\"ollmer-Ito calculus and removes previous restriction on the time partition sequence. We introduce a foliation structure on this path space and show that harmonic functionals may be represented as pathwise integrals of closed 1-forms.
\end{abstract}

\newpage
\tableofcontents
\clearpage

\section{Introduction}\label{sec:intro} 
\subsection{Motivation}
	Let $\pi:=(\pi_n)_{n\geq 1}$ be a sequence of interval partitions of $[0,\infty)$ and denote $Q^{\pi}$ the set of \cadlag paths with finite quadratic variation along $\pi$ in the sense of \follmer\  \cite{HF}. Then for any $f\in C^2(\R^d)$, the \ito formula holds pathwise along any path $x\in Q^{\pi}$ \cite{HF}:
	\begin{alignat}{2}	f(x(T))&=f(x(0))+ 	\int_{0}^{T}\nabla f(x(t-))dx(t)+ \frac{1}{2}\int_{0}^{T} \nabla^2 f(x(t)).d[x]^{c}(t)\\
	&+ \sum_{0\leq s \leq t} \Delta f(x(s))- \nabla f(x(s-)).\Delta x(s)\nonumber
	\end{alignat}
	where the second term  $\int_{0}^{T}\nabla f(x(t-))dx(t)$ is a  "\follmer\  integral", defined as a pointwise limit of left Riemann sums:\begin{eqnarray}\label{eq.pathwise}
		\int_{0}^{T}\nabla f(x(t-)).dx(t):=\lim_{n\to\infty}\sum_{\pi_n\ni t_{i}\leq T}\nabla f(x(t_{i}))(x(t_{i+1})-x(t_{i})),
		\end{eqnarray} without  resorting to any probabilistic notion of convergence. 
		Based on the key observation that, for any semi-martingale $X$, there exists a sequence of partitions $\pi$ such that the sample paths of $X$ lie almost surely in $Q^{\pi}$,
\follmer\  showed \cite{HF} that for any integrand of the form $\nabla f\circ X$, where $f\in C^2(\R^d)$, the pathwise integral \eqref{eq.pathwise} coincides with probability one with the \ito integral, thus providing a pathwise interpretation of the \ito stochastic integral. 

    The extension of this result to path-dependent functionals has been the focus of several recent works \cite{ananova,CF,CP2019,mania2020}.
    In particular, a  change of variable formula for for a class of regular functionals  of \cadlag paths was obtained in  \cite[Thm. 4]{CF}.
    Moreover, \cite{CF}  (see also \cite[Thm 3.2]{ananova}) establishes that, for $F\in \mathbb{C}^{1,2}(\Lambda_T)$,  one may  define a pathwise integral $\int_{0}^{T}\nabla_{x} F(t,x_{t-}).d^\pi x$  as a pointwise limit of Riemann sums as in \eqref{eq.pathwise}.
    
    The key idea  behind these results \cite{RC,CF09,CF}  can be summarised as follows \cite{RC}. First, one constructs a calculus for continuous functionals  on piecewise constant paths. Second, this calculus is extended to all \cadlag paths   using a density argument, using  piecewise-constant approximations  of paths. This second step is where topology plays a role. 
	    The original construction of the functional \ito calculus was based on the uniform topology \cite{CF09,CF,BD}. As is well known, piece-wise constant approximation of a \cadlag path under the uniform topology requires \emph{exact} knowledge of all points of discontinuity, which leads to a requirement   \cite[Rem.7]{CF} that the sequence of partitions exhausts the set $J(x)$ of discontinuity points of the path $x$:
    \begin{eqnarray} J(x):= \{ t\in [0,\infty), \quad x(t-)\neq x(t)\}\subset\liminf_{n}\pi_n.\label{eq:assumption}   \end{eqnarray}  
    This condition, which links the partition with the path,   is not required  for \follmer's \cite{HF} results, but plays a key role in the proof of \cite[Thm. 4]{CF}.

The following result, whose proof is given in Section \S~\ref{sec:append}, shows that this condition \eqref{eq:assumption} is restrictive and need not be satisfied, even for semimartingales:  
\begin{Prop}\label{prop:invalid} There exists a semi-martingale $X$ such that for {\it any} partition sequence $\pi$, 
$\mathbb{P}(J(X)\subset\liminf_{n}\pi_n)=0$.\end{Prop}


	A related issue is  the  differentiability and regularity of the pathwise integral. 
The \follmer\  integral $\mathbb{I}:(t,x)\mapsto \int_{0}^{t}\nabla_{x}F.d^{\pi}x$, which is a central object in the pathwise \ito calculus,  is not continuously  differentiable in the sense of \cite{CF}, even for $F\in \mathbb{C}^{1,2}(\Lambda_T)$.

To address these issues one needs to replace the uniform topology with another topology. 
Unfortunately, the usual topologies on the Skorokhod space $D$ \cite[s5]{AJ} do not fit this purpose. For example the pointwise evaluation map\begin{alignat*}{1}F(x):=x(t)\end{alignat*}is not J$_1$ continuous on $D$ \cite[VI. 2.3]{JS} and the same applies to all weaker topologies. It may thus be a lost cause to obtain a functional calculus built on top of weak topologies on $D$.
 
  In this work we circumvent these obstacles by introducing a new topology on the space $D$ of \cadlag paths.
  The F\"ollmer pathwise integral and the pathwise quadratic variation functional are shown to be  continuous functionals with respect to this topology.
  We define a class   of continuously differentiable functionals with respect to this topology and derive  change of variable formulas for such functionals without requiring  the
restrictive  condition \eqref{eq:assumption}.
In the case of paths with finite quadratic variation along a partition sequence, our change of variable formula extends  results \cite{ananova,CF,HF,hirai} on the F\"ollmer-Ito calculus and relaxes previous assumptions relating the partition sequence to the discontinuities of the underlying path.
In particular we obtain a pathwise identity of \ito (Theorem \ref{thm:isometry}) in the spirit of Beiglb\"ock and Siorpaes'  pathwise Burkholder-Davis-Gundy inequality \cite{PS}. 

Pathwise integration concepts and It\^o-type change of variable formulas have been obtained by Cont \& Perkowski \cite{CP2019} using an extension of \follmer's ideas to paths with p-th order variation and by Friz \& Zhang \cite{frizzhang} using rough path theory. 
In contrast to these results, we define pathwise integrals as limits of (left-)Riemann sums, which naturally appear in applications, not compensated Riemann sums, and we are able to treat a greater class of functionals, notably including \follmer\  integrals.


\subsection{Outline}

	After introducing some  definitions and notations in Section \ref{sec:preliminaries} we prove, in section~\ref{sec:qv},   a new limit theorem  which is useful for studying functionals   involving quadratic variation. In section~\ref{sec:pi}, we introduce a new topology the space of \cadlag paths, discuss its  relation with other well-known topologies and give examples of continuous functionals for this topology. 
In section~\ref{sec:c1k}, we introduce classes of smooth causal functionals and discuss their properties. In particular, we introduce a class of functionals which are shown to satisfy a pathwise analogue of the martingale property (Theorem \ref{thm:fair}).

Section~\ref{sec:ito} discusses pathwise integration and functional change of variable formulas. We show in particular that pathwise integrals may be defined for class $\M$ functionals without any condition on the variation index ($p$-variation) of the underlying path.
Section~\ref{sec:ftc} discusses in more detail the case of functionals of \cadlag paths with finite quadratic variation and the relation of class $\M$ functionals to a class of path-dependent partial differential equations.

\section{Preliminaries}\label{sec:preliminaries}
	\subsection{Notations}\noindent

    Denote by $D_{m}$  the Skorokhod space of $\R^{m}$-valued \cadlag functions \begin{eqnarray*}t\longmapsto x(t):=(x_1(t),\ldots,x_{m}(t))'\end{eqnarray*} on $\R_{+}:=[0,\infty)$. Denote $\SS_m$ (resp. $BV_m$) the subset of step functions (resp. locally bounded variation functions) in $D_{m}$. For $m=1$, we will omit the subscript $m$. By convention, $x(0-):=x(0)$ and $\Delta x(t):=x(t)-x(t-)$. we denote by by $x_t\in D_{m}$ (resp. $x_{t-}\in D_{m}$) the path $x\in D_{m}$ stopped at $t$ (resp. $t-$):
    $$ x_t(s) =x(s\wedge t),\qquad x_{t-}(s) =x(s) 1_{s<t} + x(t-) x(s) 1_{s\geq t}.$$
   We equip $(D_{m},\d)$    with a   metric $\d$ which induces the Skorokhod (a.k.a. \textsc{J}$_1)$ topology. 

		Let $\pi:=(\pi_n)_{n\geq 1}$ be a fixed sequence of partitions $\pi_n=(t^{n}_0,...,t^{n}_{k_n})$ of $[0,\infty)$ into intervals $0=t^{n}_0<...<t^{n}_{k_n}<\infty$ such that $t^{n}_{k_n}\to\infty$, with vanishing mesh $|\pi_n|=\sup_{i=1..k_n} |t^n_{i}-t^n_{i-1}|\to 0$ on compacts. By convention, $\max(\emptyset\cap\pi_n):=0$, $\min(\emptyset\cap\pi_n):=t^{n}_{k_n}$.  
		
		We denote \begin{eqnarray}t'_n:=\max\{t_i<t|t_i\in\pi_n\},\qquad
x^{n}&:=&\sum_{t_i\in\pi_n}x(t_{i+1})\I_{[t_i,t_{i+1})} \label{eq.tprime}
\end{eqnarray}and by $x^{(n)}$ the (continuous) piecewise-linear approximations of $x$ along $\pi_n$. 

		We denote $Q^{\pi}_m\subset D_m$ the subset of \cadlag paths with finite quadratic variation along $\pi$, defined as follows: 
\begin{Def}[Quadratic variation along a sequence of partitions]\label{def:qv}
We say that $x\in D_m$ has finite quadratic variation along $\pi$  if the  sequence of step functions:\begin{eqnarray*}
q_n(t):=\sum_{\pi_n\ni t_i\leq t}(x({t_{i+1}})-x({t_{i}}))(x({t_{i+1}})-x({t_{i}}))'
\end{eqnarray*}converges in the Skorokhod topology. The limit $[x]_\pi:=\left([x_{i},x_{j}]_\pi\right)_{1\leq i, j\leq m}\in D_{m\times m}$ is called the quadratic variation of $x$ along $\pi$.\end{Def}
In the sequel, we shall fix such a sequence of partitions $\pi$ and drop the subscript $\pi$ unless we want to emphasize the dependence on $\pi$.

As shown in \cite[Thm. 3.6]{CC}, Definition \ref{def:qv} is equivalent to the one given by \follmer \cite{HF}:
\begin{Prop}[\cite{CC}]\label{prop:qv}
Let $x\in D_m$, then $x\in Q^{\pi}_{m}$ if and only if $x_i, x_i+x_j \in Q^{\pi}$. If $x\in Q^{\pi}_m$, then we have the polarisation identity \begin{alignat}{2}
[x_{i},x_{j}](t)&=\frac{1}{2}\left([x_{i}+x_{j}]-[x_{i}]-[x_{j}]\right)(t)\in BV\nonumber\\
&=[x_{i},x_{j}]^{c}(t)+\sum_{s\leq t}\Delta x_i(s) \Delta x_j(s)\label{eq:qv}
\end{alignat}
\end{Prop}

We set $\lim_{n}a_n:=\infty$ whenever a real sequence $(a_n)$ does not converge. For real-valued matrices of equal dimension, we write $\langle\cdot,\cdot\rangle$ to denote the Frobenius inner product and $|\cdot|$ to denote the Frobenius norm. If $f$ (resp. $g$) are  $\R^{m\times m}$-valued functions on $[0,\infty)$, we write\begin{eqnarray}
		\int_{0}^{t} fdg:=\sum_{i,j}\int_{0}^{t}f_{i,j}(s-)dg_{i,j}(s)
		\end{eqnarray}whenever the RHS makes sense. If $x\in Q^{\pi}_m$ and $f\in C^2(\R^{m})$, we write \begin{eqnarray*}
		\int_{0}^{t}(\nabla f\circ x) d^\pi x:=\int_{0}^{t}\nabla f(x(s-)) d^\pi x(s)\end{eqnarray*}to denote the \follmer\  integral \cite{HF}, defined as a pointwise limit of left Riemann sums along $\pi$. The superscript $\pi$ may be dropped in the sequel as $\pi$ is fixed throughout.

\subsection{Quadratic Riemann sums}\label{sec:qv}\noindent 
	In this section, we focus on paths with finite quadratic variation along a sequence of partitions and  extend certain limit theorems obtained in  	\cite{CF}  for the convergence of 'quadratic Riemann sums' (in particular \cite[Lemma 12]{CF}) to a more general setting.
	The main result of this section is Theorem~\ref{thm:qs}, which is a key ingredient in the proof of change of variable formula for functionals of paths with quadratic variation.
	

 \noindent
The following result \cite[Lemma 2.2]{CC} will be useful in the sequel: 
\begin{Lem}\label{lem:vague1}
Let $v_n, v$ be non-negative Radon measures on $\R_{+}$ and $J$ be the set of atoms of $v$. Then $v_n\rightarrow v$ vaguely on $\R_{+}$ if and only if $v_n\rightarrow v$ weakly on $[0,T]$ for every $T\notin J$.\end{Lem}
 
\begin{Lem}\label{lem:d[x]}
Let $x\in Q^{\pi}$, $\mu=d[x]$ be the Radon measure associated with $[x]$. For every $[0,T]$, $T_n:=\max\{t_i<T|t_i\in\pi_n\}$, $T_{n+1}:=\min\{t_i\geq T|t_i\in\pi_n\}$. Define a sequence of non-negative Radon measures on $\R_{+}$ by\begin{eqnarray*}
\mu_n([0,T])&:=&\sum_{t_{i}\in\pi_n}(x(t_{i+1})-x(t_{i}))^2\delta_{t_{i+1}}([0,T))+(x(T_{n+1})-x(T_{n}))^2.
\end{eqnarray*}Then\begin{itemize}
\item[(i)] $\xi_n:=\sum_{t_i\in\pi_n}(x(t_{i+1})-x(t_{i}))^2\delta_{t_{i}}\longrightarrow \mu $ vaguely on $\R_{+}$,
\item[(ii)] $\mu_n\longrightarrow \mu $ vaguely on $\R_{+}$.
\end{itemize}
\end{Lem}

\begin{proof}
(i) follows from \cite[Thm. 2.7]{CC}. By Lemma~\ref{lem:vague1}, we may assume $T$ to be a continuity point of $d[x]$. Let $f$ be a continuous function on $[0,T]$. If $T=0$, then $\mu_n(\{0\})\equiv d[x](\{0\})=0$. If $T>0$, observe that $\xi_n([0,T))\longrightarrow d[x]([0,T))$ (by (i)), $f$ is uniform continuous on $[0,T]$ and that $x$ is right-continuous. Let $T'_{n+1}:=\min\{t_i>T|t_i\in\pi_n\}$, it follows that for sufficiently large $n$\begin{eqnarray*}
\left|\int_{0}^{T}fd\xi_{n}-\int_{0}^{T}fd\mu_{n}\right|&\leq&\sum_{\pi_n\ni t_{i}<T}|f(t_i)-f(t_{i+1}\wedge T)|(x(t_{i+1})-x(t_i))^2\\
&+&f(T)(x(T'_{n+1})-x(T_{n+1}))^2\\
&\leq&\sup_{t_i\in\pi_n\cap[0,T]}|f(t_i)-f(t_{i+1}\wedge T)|\xi_n([0,T))\\
&+&\|f\|_T(x(T'_{n+1})-x(T_{n+1}))^2\longrightarrow 0.
\end{eqnarray*}\end{proof}

\begin{Lem}\label{lem:vague2}
Let $(v_n,n\geq 1)$ be a sequence of non-negative Radon measures on $\R_{+}$ converging vaguely to a Radon measure $v$ and $J$ be the set of atoms of $v$. If for every $T\in J$, there exists a sequence $(T_n)$ in $\R_{+}$, $T_n\uparrow T$ such that\begin{eqnarray}v_n(\{T_n\})\longrightarrow v(\{T\}),\label{eq:vague2}\end{eqnarray}then $v_n\longrightarrow v$ weakly on $[0,T]$ for all $T\geq 0$.\end{Lem}

\begin{proof}
For every $T\geq 0$, $\tilde{v}_n([0,T]):=v_n([0,T])-v_n(\{T_n\})$ and $\tilde{v}([0,T]):=v([0,T])-v(\{T\})$. If $T\notin J$, the claim follows immediately from Lemma~\ref{lem:vague1}. Thus, we may assume $T\in J$. If $T=0\in J$, then $T_n\equiv 0$. Let $T>0$ and $f\in C\left([0,T],\|\cdot\|_{\infty}\right)$. Since $f=(f)^{+}-(f)^{-}$, we may take $f\geq 0$ and for sufficiently small $\epsilon>0$, we define the following extensions:\begin{eqnarray*}\overline{f}^{\epsilon}(t)&:=&f(t)\I_{[0,T]}(t)+f(T)\left(1+\frac{T-t}{\epsilon}\right)\I_{(T,T+\epsilon]}(t)\\
\underline{f}^{\epsilon}(t)&:=&f(t)\I_{[0,T-\epsilon]}(t)+f(T)\left(\frac{T-t}{\epsilon}\right)\I_{(T-\epsilon,T]}(t)
,\end{eqnarray*}then $\overline{f}^{\epsilon}$, $\underline{f}^{\epsilon}\in\C_K([0,\infty))$, $0\leq\underline{f}^{\epsilon}\leq f\I_{[0,T]}\leq\overline{f}^{\epsilon}\leq\|f\|_{\infty}$. and we have \begin{eqnarray*}\int_{0}^{\infty}\underline{f}^{\epsilon}d\tilde{v}_{n}\leq\int_{0}^{T} f d\tilde{v}_{n}\leq\int_{0}^{\infty}\overline{f}^{\epsilon}d\tilde{v}_n.\end{eqnarray*}Since $v_n \rightarrow v$ vaguely and (\ref{eq:vague2}) holds, we obtain \begin{eqnarray*}0&\leq&\limsup_{n}\int_{0}^{T} f d\tilde{v}_{n}-\liminf_{n}\int_{0}^{T} f d\tilde{v}_{n}\leq\int_{0}^{\infty}\overline{f}^{\epsilon}-\underline{f}^{\epsilon}d\tilde{v}\\
&\leq&f(T)\left(v\left([T-\epsilon,T+\epsilon]\right)-v(\{T\})\right)\stackrel{\epsilon}{\longrightarrow}0,\end{eqnarray*}hence by monotone convergence\begin{eqnarray*}\lim_{n}\int_{0}^{T} f d\tilde{v}_{n}=\lim_{\epsilon}\int_{0}^{\infty}\underline{f}^{\epsilon}d\tilde{v}=\int_{0}^{T}fd\tilde{v}.\end{eqnarray*}By (\ref{eq:vague2}), it follows $\lim_{n}\int_{0}^{T} f dv_n=\int_{0}^{T}fdv$.\end{proof}

\begin{Lem}\label{lem:vague3}
Let $(v_n,n\geq 1)$ be a sequence of non-negative Radon measures on $\R_{+}$ converging vaguely to a Radon measure $v$ and $J$ be the set of atoms of $v$. Let $f_n, f$ be real-valued left-continuous functions on $\R_{+}$ and $J$ be the set of atoms of $v$. If\begin{itemize}

\item[(i)]  for every $T\in J$ there exists a sequence $(T_n)\in[0,T)$ with $T_n\uparrow T$ such that $v_n(\{T_n\})\longrightarrow v(\{T\})$, and
\item[(ii)] $(f_n)$ is locally bounded and converges pointwise to $f$,\end{itemize}
\noindent then for every $T\geq 0$,\begin{eqnarray*}
\int_{0}^{T}f_ndv_n\longrightarrow\int_{0}^{T}fdv.
\end{eqnarray*}\end{Lem}

\begin{proof}
Let $v=v^c+v^d$ be the Lebesgue decomposition of $v$ into an absolutely continuous part $v^c$ and a singular (discrete) measure $v^d$. By (i) and Lemma~\ref{lem:vague2}, we immediately see that $(v_n-v^d)\longrightarrow v^{c}$ weakly for every $[0,T]$. Since $v^{c}$ has no atoms, by an application of \cite[Lemma 12]{CF} we have\begin{eqnarray*} 
\int_{0}^{T}f_nd(v_n-v^d)\longrightarrow\int_{0}^{T}fdv^c.
\end{eqnarray*}By (ii) and dominated convergence, the proof is complete.
\end{proof}

\begin{Thm}\label{thm:qs}
Let $x\in Q^{\pi}$, $f_n, f$ be real-valued left-continuous functions on $\R_{+}$ such that $(f_n)$ is locally bounded and converges pointwise to $f$ on $\R_{+}$. Then for any $T>0,$\begin{alignat*}{2}
(i) &\sum_{\pi_n\ni t_i\leq T}f_n(t_{i})(x(t_{i+1})-x(t_i))^2&\longrightarrow\int_0^Tfd[x].\\
(ii) &\sum_{\pi_n\ni t_i\leq T}f_n(t_{i+1}\wedge T)(x(t_{i+1})-x(t_i))^2&\longrightarrow\int_0^Tfd[x].\\
(iii) &\sum_{\pi_n\ni t_i<T} f_n(t_{i})(x(t_{i+1})-x(t_i))^2&\longrightarrow\int_0^Tfd[x].\\
(iv) &\sum_{\pi_n\ni t_i<T}f_n(t_{i+1}\wedge T)(x(t_{i+1})-x(t_i))^2&\longrightarrow\int_0^Tfd[x].
\end{alignat*}
\end{Thm}

\begin{proof}
If $T=0$, then by (\ref{eq:qv}) and that $x$ is right-continuous and has no discontinuity at $T=0$, the claims follow. If $T>0$, put $T_n:=\max\{t_i<T|t_i\in\pi_n\}$, $T_{n+1}:=\min\{t_i\geq T|t_i\in\pi_n\}$, $T'_{n+1}:=\min\{t_i>T|t_i\in\pi_n\}$, then $T_n\uparrow T$ and by Lemma~\ref{lem:d[x]}, we observe that\begin{eqnarray*}
\xi_n(\{T_n\})=(x(T_{n+1})-x(T_n))^2\longrightarrow d[x](\{T\}),\\
\mu_n(\{T\})=(x(T_{n+1})-x(T_n))^2\longrightarrow d[x](\{T\}),
\end{eqnarray*}and that\begin{alignat*}{2}
&\sum_{\pi_n\ni t_i<T}f_n(t_{i})(x(t_{i+1})-x(t_i))^2&=&\int_0^{T}f_nd\xi_n\\
&&-&f(T_{n+1})(x(T'_{n+1})-x(T_{n+1}))^2,\\
&\sum_{\pi_n\ni t_i\leq T}f_n(t_{i+1}\wedge T)(x(t_{i+1})-x(t_i))^2&=&\int_0^{T}f_nd\mu_n\\
&&+&f(T)(x(T'_{n+1})-x(T_{n+1}))^2.
\end{alignat*}By the right continuity of $x$, Lemma~\ref{lem:d[x]} and Lemma~\ref{lem:vague3}, the proof is complete.\end{proof}

As a consequence of Prop.~\ref{prop:qv} and Thm.~\ref{thm:qs} we have:
\begin{Cor}[Multidimensional paths]\label{cor:qs}
Let $x\in Q^{\pi}_m$, $f_n, f: \R_{+}\mapsto \R^{m\times m}$ be  left-continuous functions with $(f_n)$   locally bounded and converging pointwise to $f$ on $\R_{+}$. Then\begin{alignat*}{2}
(i) &\sum_{\pi_n\ni t_i\leq T}\langle f_n(t_{i}),(x(t_{i+1})-x(t_i))(x(t_{i+1})-x(t_i))'\rangle&\longrightarrow\int_0^T fd[x]\\
(ii) &\sum_{\pi_n\ni t_i\leq T}\langle f_n(t_{i+1}\wedge T),(x(t_{i+1})-x(t_i))(x(t_{i+1})-x(t_i))'\rangle&\longrightarrow\int_0^T fd[x]\end{alignat*}for every $T\geq 0$. In particular, the convergence also holds if the sum is replaced by $\sum_{\pi_n\ni t_i<T}$.\end{Cor}
  
\begin{Rem}\label{rem:qv}
$t\longmapsto\int_0^t fd[x]$ is in $BV$ and has Lebesgue decomposition: \begin{eqnarray*}\int_0^t fd[x]=\int_0^t fd[x]^c+\sum_{s\leq t}\langle f(s-),\Delta x(s)\Delta x(s)'\rangle.\end{eqnarray*}
\end{Rem}

\section{Continuous functionals}\label{sec:pi}\noindent 	We now construct a topology on suitable subsets of \begin{alignat*}{1}E:=\R_{+}\times D_m,\end{alignat*} for which the \follmer\  integral $x \mapsto \int_0^T \phi.d^\pi x$ will be a continuous functional of the integrator $x$. 
	
	\subsection{Domains for causal functionals}
	
	We are interested in  \emph{causal} (non-anticipative) functionals \cite{RC,fliess}, whose natural domain of definition is a set of stopped paths \begin{eqnarray*}\{(t,x_t)|t\in\R_{+}, x\in \Omega\}\subset E,\end{eqnarray*}for a suitable set of paths $\Omega\subset D_m$, where $x_t=x(t\wedge .)$ \cite{CF}. 
	
	In order to deploy our functional calculus on such functionals we require $\Omega\subset D_m$ to be closed under certain operations:
	\begin{itemize}
	    \item stopping: $x\in \Omega\Longrightarrow \forall t\geq 0, \ x_t=x(t \wedge .)\in \Omega. $\item vertical perturbations, in order to define the vertical (Dupire) derivative:\begin{eqnarray*}x\in \Omega\Longrightarrow x_{t}+e\I_{[t,\infty)}\in \Omega,\end{eqnarray*}
	    
	    \item piecewise constant approximation along $\pi$.
	\end{itemize}
	 We will call {\it \closed} a set of paths stable under these operations:
	\begin{Def}[Generic sets of paths]\label{def:closed}
	A non-empty subset $\Omega\subset D_m$ is called \emph{\closed} if it satisfies: 
	\begin{itemize}
	\item[i)] Stability under piecewise constant approximation along $\pi$: For every $x\in \Omega$, $T> 0$, $\exists N\in\N$;  $x^n_T\in \Omega,\quad\forall n\geq N$.
	\item[ii)] Stability under vertical perturbation:  For every $x\in\Omega, t\geq  0$,  there exists a convex neighbourhood $\U$ of $0$ 
	such that \begin{eqnarray*}-\Delta x(t)\in \U\qquad {\rm and}\quad x_{t}+e\I_{[t,\infty)}\in \Omega,\quad \forall e\in \U.\end{eqnarray*}
    \end{itemize}
	We will call a \emph{domain}  a set $\Lambda$ of stopped paths of the form  \begin{eqnarray*}\Lambda:=\{(t,x_t)|t\in\R_{+}, x\in \Omega\}\end{eqnarray*} where $\Omega\subset D_m$ is \closed.
	\end{Def}
	
	\begin{Rem}\label{rem:closed}
     Def.~\ref{def:closed}(ii) implies that $-\U$ is a convex neighbourhood of $0$ containing $\Delta x(t)$ such that\begin{eqnarray*} x_{t-}+e\I_{[t,\infty)}\in \Omega,\quad \forall e\in -\U.\end{eqnarray*}
	\end{Rem}
	
	\begin{Eg}
	$\SS_m$, $BV_m$, $Q^{\pi}_m$, ${Q^{\pi}_m}^+$ (i.e. positive paths in $Q^{\pi}_m$) and $D_m$ are all \closed\ sets. If $\Omega$ is \closed, then \begin{eqnarray*}
	\Omega_{a}^{b}:=\{x\in \Omega| a<x_{i}<b\}
	\end{eqnarray*}for all constants $a,b$ are all \closed. Subsets of continuous paths are not \closed.
	\end{Eg}
	
	\begin{Eg}\label{Eg:base}
    Let $\Omega$ be \closed{}. Then $\Omega\cap Q^{\pi}_m$ is \closed.
	\end{Eg}
	
	\begin{proof}
	We observe $\SS_m\subset Q^{\pi}_m$ and if $x\in Q^{\pi}_m$, then $x+\SS_m\in Q^{\pi}_m$.
	\end{proof}


	On $E$, there already exists two well known (product) topologies,   generated by the standard topology on $\R_{+}$ and local uniform (resp. the Skorokhod J$_1$) topology on $D_m$. On a domain $\Lambda\subset E$, we define the uniform (U) and J$_1$ topologies as the corresponding  topology induced on $\Lambda$.  

\begin{Rem}\label{rem:uniform}
Every J$_1$-continuous functional is U-continuous: the local uniform topology is strictly finer than the J$_1$ topology on $D_m$ \cite[VI]{JS}.
\end{Rem}

We will now show that, if  $\Omega$ is 'rich enough' to contain a path with non-zero quadratic variation as well as its piecewise-linear approximations  along $\pi$, then important examples  of functionals such as quadratic variation or the \follmer\  integral fail to be continuous on $\Omega$ in the uniform topology.
We use the following assumption:
	\begin{assumption}\label{def:rich}
 $\Omega$ is a \closed{} subset 
 and contains a path
 $x\in Q^{\pi}_m$ with $[x]_\pi$ continuous and strictly increasing, as well its piecewise linear approximations along $\pi$: $$
\exists N\in\N, \forall n\geq N, x^{(n)}\in \Omega,$$ where $x^{(n)}$ denotes the piecewise-linear approximation of $x$ along $\pi_n$.
	\end{assumption}

	\begin{Eg}
	$Q^{\pi}_m$ and ${Q^{\pi}_m}^+$ satisfy Assumption \ref{def:rich}, $\SS_m$ and $BV_m$ do not.
	\end{Eg}

\begin{Lem}\label{lem:compare} 
Let $\Omega$ satisfy Assumption \ref{def:rich} and $\Lambda=\{(t,x_t)|t\in\R_{+}, x\in \Omega\}$. Then the functionals 
$$ F(t,x_t):=|[x](t)|  \qquad
 G(t,x_t):=\int_{0}^t2x dx$$ are not U-continuous on $\Lambda$.
\end{Lem}

\begin{proof}
If $\Omega$ satisfies Assumption \ref{def:rich}, there exists $T>0$, continuous $x,x^{(n)}\in\Omega$ such that $\left|[x](T)\right|>0.$
Since $x^{(n)}_T\longrightarrow x_T$ in the local uniform topology on $[0,\infty)$, it follows   that \begin{eqnarray*}
(T,x^{(n)}_T)\stackrel{\textsc{U}}{\longrightarrow}(T,x_{T})
\end{eqnarray*}on $\Lambda$. Since $x^{(n)}_T$ is a continuous function of bounded variation on $[0,\infty)$, it follows that \begin{eqnarray*}\label{eq:WZK2}|[x^{(n)}](T)|=0,\quad\forall n\geq 1\end{eqnarray*} so $F$ is not U-continuous. Using the above and the fact  that $x,x^{(n)}\in Q^{\pi}_m$, we obtain by an application of the pathwise \ito formula \cite{HF}: \begin{alignat*}{2}
&\lim_{n}\left|\int_{0}^T2xdx - \int_{0}^T2x^{(n)}dx^{(n)}\right|\\
=&\lim_{n}\left||x(T)|^2-|x(0)|^2-tr\left([x](T)\right)-\left(|x^{(n)}(T)|^2-|x^{(n)}(0)|^2\right)\right|\\
=&tr\left([x](T)\right)>0,\end{alignat*}hence $G$ is not U-continuous on $\Lambda$.
\end{proof}

    We shall now define a new topology on a domain $\Lambda$ for which these examples of functionals will be continuous. 
\subsection{The $\pi$-topology}

\begin{Def}[The $\pi-$topology]\label{def:pi}\noindent\\
For every $t\in\R_{+}, x\in \Omega$, we define $t'_n:=\max\{t_i<t|t_i\in\pi_n\}$ and
\begin{eqnarray}
x^{n}&:=&\sum_{t_i\in\pi_n}x(t_{i+1})\I_{[t_i,t_{i+1})}.\label{eq:x_n}
\end{eqnarray} Denote $\mathfrak{X}$ the set of functionals $F:\Lambda\longmapsto \R$ satisfying:

\begin{alignat*}{3}
1.&(a)\qquad  \lim_{s\uparrow t; s\leq t}F(s,x_{s-})=F(t,x_{t-}),\\ 
  &(b)\qquad \lim_{s\uparrow t; s<t}F(s,x_{s})=F(t,x_{t-}),\\
  &(c)\qquad t_n\longrightarrow t; t_n\leq t'_n \Longrightarrow F(t_n,x^{n}_{t_n-})\longrightarrow F(t,x_{t-}),\\
  &(d)\qquad t_n\longrightarrow t; t_n<t'_n \Longrightarrow F(t_n,x^{n}_{t_n})\longrightarrow F(t,x_{t-}),\\
\newline\\
2.&(a)\qquad \lim_{s\downarrow t; s\geq t}F(s,x_{s})=F(t,x_{t}),\\
  &(b)\qquad \lim_{s\downarrow t; s>t}F(s,x_{s-})=F(t,x_{t}),\\
  &(c)\qquad t_n\longrightarrow t; t_n\geq t'_n\Longrightarrow F(t_n,x^{n}_{t_n})\longrightarrow F(t,x_t),\\
  &(d)\qquad t_n\longrightarrow t; t_n>t'_n\Longrightarrow F(t_n,x^{n}_{t_n-})\longrightarrow F(t,x_{t}),\\
\end{alignat*} for all $(t,x_t)\in\Lambda$. The initial topology generated by $\mathfrak{X}$ on $\Lambda$ is called the $\pi-$topology.\end{Def}
We note that the definition of this topology depends on the partition sequence $\pi$.
\begin{Rem}\label{rem:uniform2}
Every U-continuous functional satisfies Def.~\ref{def:pi}.1(a),(b) and 2(a),(b). 
\end{Rem}

\begin{Def}[Continuous functionals]\label{def:cont}\noindent\\
 We  denote $C(\Lambda)$ the set of functionals $F:\Lambda\longmapsto \R$ that are continuous with respect to the $\pi-$topology.

$F$ is called \emph{left- (resp. right-) continuous} if it satisfies property 1 (resp. property 2) in Definition ~\ref{def:pi}.
\end{Def}
 \begin{Rem}\label{rem:criteria} Since $$z_n\stackrel{\Lambda}{\longrightarrow}z\Longleftrightarrow F(z_n)\rightarrow F(z) \quad\forall F\in\mathfrak{X},$$ we have $C(\Lambda)\subset\mathfrak{X}$  so in fact
$C(\Lambda)=\mathfrak{X}$.
\end{Rem}

The following concept was introduced in \cite{CF} under the name 'predictable functional'; we redefine it here without any reference to measurability considerations:
\begin{Def}[Strictly causal functionals]\label{def:causal}\noindent\\
For $F:\Lambda \to \mathbb{R}^d$  denote $F_{-}(t,x_t)=F(t,x_{t-})$. $F$ is  \emph{strictly causal} if $F=F_{-}$.
\end{Def}

The following lemma follows from Def.~\ref{def:pi}.1(a) and (b) and Def.~\ref{def:pi}.2(a) and(b).
\begin{Lem}[Pathwise regularity]\label{lem:cont}Let $F:\Lambda\to \mathbb{R}^d$ and $x\in \Omega$.
\begin{itemize}
\item[(i)] If $F$ is left-continuous, then $t\longmapsto F_{-}(t,x_{t})$ is left-continuous and $t\longmapsto F(t,x_{t})$ has left limits.
\item[(ii)] If $F$ is right-continuous, then $t\longmapsto F(t,x_{t})$ is right-continuous and $t\longmapsto F_{-}(t,x_{t})$ has right limits.
\item[(iii)] If $F$ is continuous, then $t\longmapsto F_{-}(t,x_{t})$ (resp. $t\longmapsto F(t,x_{t})$) is \caglad (resp. \cadlag) and its jump at time $t$ is equal to $\Delta F(t,x_{t})$.
\end{itemize}
\end{Lem}

\begin{Eg}\label{eg:cont} Assume $\Omega\subset Q^{\pi}_m$. Then the functionals
\begin{itemize}
\item[(i)] {\makebox[4cm][l]{$F(t,x_t):=f(x(t))$;} $f\in C(\R^{m})$,}
\item[(ii)] {\makebox[4cm][l]{$F(t,x_t):=f([x](t))$;} $f\in C(\R^{m\times m})$,}
\item[(iii)] {\makebox[4cm][l]{$F(t,x_t):=\int_{0}^t (f\circ x) d[x]$;} $f\in C(\R^{m},\R^{m\times m})$,}
\item[(iv)] {\makebox[4cm][l]{$F(t,x_t):=\int_{0}^t(\nabla f\circ x) dx$;} $f\in C^2(\R^{m})$,}
\end{itemize} belong to $C(\Lambda)$.
\end{Eg}

\begin{proof} In the light of Remark \ref{rem:criteria},
 $F$ is continuous if and only if $F$ satisfies Def.~\ref{def:pi} for all $(t,x)\in \Lambda$. Since conditions Def.~\ref{def:pi}.1(a),(b)  and 2(a),(b) are easy to verify, we focus on Def.~\ref{def:pi}.1(c),(d)  and 2(c),(d). (i) is trivial. For (ii), we first remark from Def.~\ref{def:qv} and (\ref{eq:qv}) that\begin{alignat}{2}
q_n&\stackrel{\textsc{J$_1$}}{\longrightarrow}[x];\nonumber\\
\Delta q_n(t'_n)=\Delta x^n(t'_n)\Delta x^n(t'_n)'&\longrightarrow\Delta x(t)\Delta x(t)'=\Delta[x](t).\label{eq:catch}
\end{alignat}Since $
[x^n](t)=q_n(t)$ and by (\ref{eq:catch}), if $t_n\longrightarrow t$, the limits of $q_n(t_n)$ and $q_n(t_n-)$ are readily determined according to the rules laid down in \cite[s4.2]{CC} and (ii) immediately follows from the continuity of $f$.  

	To show (iii) and (iv), it is suffice to assume $t_n\longrightarrow t$; $t_n\geq t'_n$ (i.e. the other criteria follow similar lines of proof, see \cite[s4.2]{CC}). By (\ref{eq:catch}) and \cite[s4.2]{CC}\begin{eqnarray}
	|q_n(t_n)-q_n(t'_n)|\longrightarrow 0.\label{eq:tight}
	\end{eqnarray}A closer look at (iii), combined with Corollary~\ref{cor:qs}, leads to\begin{alignat*}{2}
&F(t_n,x^n_{t_n})&=&\int_{0}^{t_n} (f\circ x^n)d[x^n]\\
&&=&\sum_{\pi_n\ni t_i<t}\langle f(x({t_i})),(x(t_{i+1})-x(t_{i}))(x(t_{i+1})-x(t_{i}))'\rangle\longrightarrow F(t,x_t)\\
&&+&\sum_{\pi_n\ni t_i\in(t'_n,t_n]}\langle f(x({t_i})),(x(t_{i+1})-x(t_{i}))(x(t_{i+1})-x(t_{i}))'\rangle.
\end{alignat*}By (\ref{eq:tight}) and that $f\circ x$ is locally bounded on $\R_{+}$, we see that the absolute value of the last term is bounded by $
\text{const}|q_n(t_n)-q_n(t'_n)|\longrightarrow 0$. 

 For (iv), from the properties of the \follmer\  integral \cite{HF}, we first observe that\begin{alignat*}{2}
&F(t_n,x^n_{t_n})&=&\int_{0}^{t_n}\nabla (f\circ x^n) dx^n\\
&&=&\sum_{\pi_n\ni t_i<t}\nabla f(x({t_i}))\cdot(x(t_{i+1})-x(t_{i}))\longrightarrow F(t,x_t)\\
&&+&\sum_{\pi_n\ni t_i\in(t'_n,t_n]}\nabla f(x({t_i}))\cdot(x(t_{i+1})-x(t_{i})).
\end{alignat*}Define  $\underline{t_n}:=\min\{t_i>t'_n|t_i\in\pi_n\}$, $\overline{t_n}:=\min\{t_i>t_n|t_i\in\pi_n\}$ and note that $\overline{t_n}\geq\underline{t_n}\geq t$, hence  \begin{eqnarray*}|f(x(\overline{t_n}))-f(x(\underline{t_n}))|\longrightarrow 0.\end{eqnarray*} 
Applying a second order Taylor expansion to $f$ and using (\ref{eq:tight}), we obtain\begin{alignat*}{2}
\left|\sum_{\pi_n\ni t_i\in(t'_n,t_n]}\nabla f(x({t_i}))\cdot(x(t_{i+1})-x(t_{i}))\right|&\leq|f(x(\overline{t_n}))-f(x(\underline{t_n}))|\\
&+\text{const}|q_n(t_n)-q_n(t'_n)|\longrightarrow 0.
\end{alignat*}
\end{proof}

\begin{Rem}\label{rem:x_n}
If $x\in D_m$, so are $x_{T}$ and $x_{T-}$ and the corresponding piecewise constant approximation(s) in (\ref{eq:x_n}) shall be denoted by $(x_{T})^n$ and $(x_{T-})^n$.
\end{Rem}

The following property may be derived from \cite[Lemma 12.3]{PB} and \cite[VI]{JS}:
\begin{Lem}\label{lem:approx}
Let $T\geq 0$, $x\in D_m$, then $(x_{T})^n\stackrel{\textsc{J$_1$}}{\longrightarrow}x_{T}$.
\end{Lem}

\begin{Lem}\label{lem:approx2}
Let $(t,x)\in \Lambda$, $t_n\longrightarrow t$ and denote $t'_n:=\max\{t_i<t|t_i\in\pi_n\}$. Then\begin{alignat*}{3}
&(i)\qquad t_n\leq t'_n\Longrightarrow x^{n}_{t_n-}&\stackrel{\textsc{J$_1$}}{\longrightarrow}& & x_{t-}&,\\
&(ii)\qquad t_n<t'_n\Longrightarrow x^{n}_{t_n}&\stackrel{\textsc{J$_1$}}{\longrightarrow}& &x_{t-}&,\\
&(iii)\qquad t_n\geq t'_n\Longrightarrow x^{n}_{t_n}&\stackrel{\textsc{J$_1$}}{\longrightarrow}& &x_{t}&,\\
&(iv)\qquad t_n>t'_n\Longrightarrow x^{n}_{t_n-}&\stackrel{\textsc{J$_1$}}{\longrightarrow}& &x_{t}&.
\end{alignat*}
\end{Lem}

\begin{proof}
Let $t_n\leq t'_n$, by Lemma~\ref{lem:approx}, we have $(x_{t-})^n\stackrel{\textsc{J$_1$}}{\longrightarrow}(x_{t-})$. Since $x$ is \cadlag we observe\begin{eqnarray*}
\|x^n_{t_n-}-(x_{t-})^n\|_{\infty}\leq\sup_{s\in[t_n,t'_n]}|x(t_n)-x(s)|+|x(t_n)-x(t-)|\longrightarrow 0
,\end{eqnarray*}and (i) follows immediately from \cite[VI.1.23]{JS}. (ii)-(iv) follow similar lines of proof.
\end{proof}

\begin{Thm}\label{thm:compare}Let $\Omega$ satisfy Assumption \ref{def:rich}. Then:
\begin{itemize}
\item[(i)] Every \textsc{J$_1$}-continuous functional is continuous.
\item[(ii)] There exists a continuous functional which is not \textsc{U}-continuous.
\item[(iii)] There exists  \textsc{U}-continuous functionals which are not continuous.
\end{itemize}
\end{Thm}

\begin{proof}
If $F$ is J$_1$-continuous, then $F$ satisfies Def.~\ref{def:pi}.1(a),(b)  and 2(a),(b) due to Rem.~\ref{rem:uniform}  and \ref{rem:uniform2}. (i) now follows immediately from Lemma~\ref{lem:approx2}. (ii) is due to Example~\ref{eg:cont} and Lemma~\ref{lem:compare}. 

	It remains to show (iii). We first note that the U topology on $\Lambda$ is metrisable, hence sequential continuity is equivalent to continuity. Let us fix a $t_0>0$; $t_0\notin\cup_n\pi_n$, define \begin{eqnarray*}F(t,x_t):=|\Delta x_t(t_0)|\end{eqnarray*}on $\Lambda$. Observe that if $x_n\stackrel{\textsc{U}}{\longrightarrow}x$ in $D_m$ then it is well known that:
\begin{eqnarray}
\Delta x_n(s)\longrightarrow \Delta x(s)\label{eq:y_n}
\end{eqnarray}for  $s\geq 0$. In particular, if $t_n\longrightarrow t$; $x_n(\cdot\wedge{t_n})\stackrel{\textsc{U}}{\longrightarrow}x_t$ then (\ref{eq:y_n}) implies $\Delta x_n(\cdot\wedge{t_n})(s)\longrightarrow \Delta x_t(s)$ for $s\geq 0$, hence $F$ is U-continuous on $\Lambda$. 

	On the other hand, we take an $x\in \Omega_0$; $\Delta x(t_0)\neq 0$, it follows from our choice of $t_0$ that\begin{eqnarray*}
F(t_0,x^n_{t_0})=|\Delta x^n(t_0)|\equiv 0,
\end{eqnarray*}hence by Def.~\ref{def:pi}.2(c), $F$ is not continuous on $\Lambda$ and (iii) follows.\end{proof}
 So, if $\Omega$   satisfies Assumption \ref{def:rich}, Theorem \ref{thm:compare} and Remark~\ref{rem:uniform} imply that
\begin{itemize}
\item the $\pi-$topology is strictly finer than the \textsc{J$_1$} topology.
\item the $\pi-$topology and the \textsc{U} topology are not comparable.
\end{itemize}

\section{Smooth functionals}\label{sec:c1k}\noindent 

	 The change of variable formulas in \cite{HF} make use of the concepts of \emph{local boundedness} and the existence of a  \emph{modulus of continuity}. 
  In this section, we shall introduce weaker notions of \emph{boundedness} and \emph{modulus of continuity} for causal functionals  and define a corresponding notion of a $C^{1,2}$ functional on $\Lambda$, and use these notions to derive a functional change of variable formula. We then introduce  $\X(\Lambda)$ and $\M(\Lambda)$, two important subspaces of $C^{1,2}(\Lambda)$.
  
   When $\Omega\subset Q^{\pi}_m$, we will show that  functionals  such as quadratic variation and \follmer\  integrals are not only $C^{1,2}$ but also belong to class $\M$, a sub-class of \emph{infinitely differentiable} functionals.
	Recall the definition of Dupire's horizontal and vertical derivatives 
\cite{CF,CF09,BD}:
\begin{Def}[Horizontal derivative]\label{def:dt}
$F:\Lambda\longmapsto \R$ is called { differentiable in time} or \emph{horizontally differentiable} if the following limit exists for all $(t,x_t)\in\Lambda$:\begin{eqnarray*}
\D F(t,x_t):=\lim_{h\downarrow 0}\frac{F(t+h,x_t)-F(t,x_t)}{h}. 
\end{eqnarray*}\end{Def}

\begin{Def}[Vertical derivative]\label{def:dx}
$F:\Lambda\longmapsto \R$ is called \emph{vertically differentiable} if for every $(t,x_t)\in\Lambda$, the map
$f: \U_t(x)\longmapsto \R$: 
\begin{eqnarray*}e\longmapsto F\left(t,x_t+e\I_{[t,\infty)}\right)\end{eqnarray*} is differentiable at $0$.    $\nabla_{x}F(t,x_t):=\nabla_{e}f(0) $ is called the vertical derivative of $F$ at $(t,x_t)\in\Lambda$.
\end{Def}


\noindent $F$ is called  {differentiable} on $\Lambda$ if it is vertically and horizontally differentiable at every $(t,x)\in \Lambda$.
 We extend the above definitions to vector-valued maps $F:\Lambda\to \mathbb{R}^{d\times n}$ whose components $F_{i,j}$ satisfy the respective conditions.
\begin{Prop}\label{prop:causal}
A causal functional $F:\Lambda\to \mathbb{R}$ is strictly causal if and only if it is vertically differentiable   with vanishing vertical derivative.
\end{Prop}
\begin{proof}
The first assertion follows from the mean value theorem. To prove the converse, let $x \in\Omega$ and put $z:=x_t+e\I_{[t,\infty)}$ then $z_{t-}=x_{t-}$ and \begin{eqnarray*}
 \text{$F(t,x_{t}+e\I_{[t,\infty)})$}=F(t,z_t)=F_{-}(t,z_{t})=F_{-}(t,x_{t})=F(t,x_{t}),
\end{eqnarray*}by the strict causality of $F$ (Def.~\ref{def:causal}).
\end{proof}

\begin{Def}[Locally bounded functional]\label{def:bound}\noindent\\
$F:\Lambda\to \mathbb{R}$ is called \emph{locally bounded} if for every $x\in \Omega$ and $T\geq 0$, there exists $ n_0\geq N_{T}(x)$ such that the family of maps \begin{alignat*}{2}
(&t\longmapsto F(t,x^{n}_{t}), n\geq n_0)\end{alignat*}  is locally bounded on $[0,T]$.\end{Def}


\begin{Lem}\label{prop:bound}
Every continuous function on $\Lambda$ is locally bounded.
\end{Lem}

\begin{proof}
Let $F$ be continuous; if $F$ is not locally bounded, there exists $x\in \Omega$, $T\geq 0$, and a sub-sequence $(n_k)$; \begin{eqnarray}|F(t_{n_k}, x^{n_k}_{t_{n_k}})|>k,\quad\forall k\geq 1;\label{eq:bound}\end{eqnarray}$(t_{n_k})$ is bounded on $[0,T]$. 
	For ease of notation, assume $t_{n_k}\longrightarrow t\in[0,T]$ without  passing through to a sub-sequence. Observe that one can always choose another sub-sequence, bounded (either above or below) by $t'_{n_k}=\max\{t_i<t|t_i\in\pi_{n_k}\}$. Since $F$ is continuous, if $t_{n_k}<t'_{n_k}$ (resp. $t_{n_k}\geq t'_{n_k}$), then Def.~\ref{def:pi}.1(d) (resp. 2(c)) would contradict (\ref{eq:bound}) as $k\uparrow\infty$.\end{proof}

\begin{Lem}\label{lem:bound}Let $F$ be locally bounded and denote $F_-(t,x)=F(t,x_{t-})$. \begin{itemize}
\item[(i)] If $F$ is left-continuous then $F_{-}$ is locally bounded.
\item[(ii)] If $F$ is left-continuous then $t\longmapsto F_{-}(t,x_{t})$ is locally bounded.
\item[(iii)] If $F$ is right-continuous then $t\longmapsto F(t,x_t)$ is locally bounded.
\end{itemize}
\end{Lem}

\begin{proof}
Since $F$ is locally bounded, there exists a constant $K>0$ such that \begin{eqnarray*}|F(t,x^n_{t})|\leq K\end{eqnarray*}for all $t\leq T$ and all $n$ sufficiently large. If $F$ is left-continuous, then Def.~\ref{def:pi}.1(b) implies\begin{eqnarray*}\text{K}\geq\lim_{s\uparrow t; s<t}|F(s,x^n_{s})|=|F(t,x^n_{t-})|,\end{eqnarray*} so (i) follows. If $t_n\longrightarrow t; t_n<t'_n$, then by the left-continuity of $F$ (i.e. Def.~\ref{def:pi}.1(d)), \begin{eqnarray*}\text{K}\geq |F(t_n,x^{n}_{t_n})|\longrightarrow |F(t,x_{t-})|,\end{eqnarray*} so (ii) follows. If $F$ is right-continuous, then by Def.~\ref{def:pi}.2(c),\begin{eqnarray*}\text{K}\geq |F(t'_n,x^{n}_{t'_n})|\longrightarrow |F(t,x_{t})|,\end{eqnarray*} so (iii) follows.\end{proof}


\begin{Def}[Modulus of vertical continuity]\label{def:space modulus}\noindent\\
We say that a function $F$ on $\Lambda$ admits a modulus of vertical continuity if for every $x\in \Omega$, $T\geq 0$ and $r>0$ there exists an increasing function $\omega:\R_{+}\longmapsto\R_{+}$ with $\omega(0+)=0$;
\begin{eqnarray}
|F(t,x^{n}_{t-}+a\I_{[t,\infty)})-F(t,x^{n}_{t-}+b\I_{[t,\infty)})|\leq\omega(|a-b|).\label{eq:space modulus}\end{eqnarray}for all $a, b\in \U_{t-}(x^n)\cap\overline{B}_{r}(0)$, $t\leq T$ and sufficiently large $n$.\end{Def}

\begin{Eg}\label{eg:modulus}
Let $f\in C(\R_{+}\times\R^{m})$. Then $F:\Lambda\to \mathbb{R}$ defined by\\ $F(t,x_t):=f(t,x(t))$ admits a modulus of vertical continuity.
\begin{proof}
For a given $x\in \Omega$ and $T\geq 0$, $r>0$, put $\|x\|_{T}:=\sup_{t\leq T}|x(t)|$, $r_0:=\alpha\|x\|_{T}+r$; $\alpha>1$, then $f$ is uniform continuous on $[0,T]\times \overline{B}_{r_0}(0)$ and a modulus of continuity of $f$ on $[0,T]\times \overline{B}_{r_0}(0)$ is given by\begin{eqnarray*}
\omega(\delta):=\sup_{|t-s|+|u-v|\leq\delta}\left|f(t,u)-f(s,v)\right|
\end{eqnarray*}which satisfies (\ref{eq:space modulus}).
\end{proof}
\end{Eg}

\begin{Rem}\label{rem:modulus}
If $F, G$ admit moduli of vertical continuity, then $\alpha F+\beta G$ admits a modulus. If in addition, $F_{-}, G_{-}$ are locally bounded, then $FG$ admits a modulus of vertical continuity.
\end{Rem}

\begin{Lem}\label{lem:modulus}
Let $F$ be  vertically differentiable and $(\nabla_{x}F)_{-}$ be locally bounded, if $\nabla_{x}F$ admits a modulus of vertical continuity then so does $F$.
\end{Lem}

\begin{proof}
Since $F$ is vertically differentiable   and $\nabla_{x}F$ admits a modulus of vertical continuity $\omega$, by the mean value theorem and the local boundedness of $(\nabla_{x}F)_{-}$, we obtain\begin{eqnarray*}
|F(t,x^{n}_{t-}+a\I_{[t,\infty)})-F(t,x^{n}_{t-}+b\I_{[t,\infty)})|\leq\left(\omega(r)+\text{const}\right)|a-b|.\end{eqnarray*}
\end{proof}

\begin{Def}[${C}^{1,2}$ functionals]\label{def:CC12}\noindent\\
We define $C^{1,2}(\Lambda)$ as the set of continuous functionals $F\in C_\pi(\Lambda)$ such that $\D F, \nabla_{x}F$ and $\nabla^2_{x}F$ are defined on $\Lambda$ and\begin{itemize}
    \item [(i)] $\D F$  is right-continuous and locally bounded.
    \item [(ii)]  $(\nabla_{x}F)_{-}$ is left-continuous,
    \item[(iii)] $(\nabla^2_{x}F)_{-}$ is left-continuous, locally bounded and admits a modulus of   vertical continuity.\end{itemize}
    If in addition, 
$(\nabla_x F)_{-}$ is locally bounded, then we denote $F\in C_{b}^{1,2}(\Lambda)$.
\end{Def}

We now introduce two classes of functionals which, as we will observe later, play a special role in the context of stochastic analysis:
\begin{Def}[Class $\X$]\label{def:smoothX}\noindent\\
A continuous and differentiable functional $F$ is of \emph{class $\X$} if $\D F$ is right-continuous and locally bounded, $\nabla_{x}F$ is left-continuous and strictly causal.  We denote by $\X(\Lambda)$ the vector space of class $\X$ functionals.
\end{Def}

\begin{Def}[Class $\M$]\label{def:smooth}\noindent\\
A   functional $F\in \X(\Lambda)$ is of \emph{class $\M$} if $\D F=0$. We denote
$\M(\Lambda)$ the set of \emph{class $\M$} functionals and
$\M_{b}(\Lambda)$ the set of   functionals $F\in\M(\Lambda) $ whose vertical derivative $\nabla_x F$ is locally bounded.
\end{Def}

\begin{Rem}
Every functional of class $\M$ is infinitely differentiable by Prop.~\ref{prop:causal}. 
\end{Rem}
Remarks      \ref{rem:modulus}, Lemma~\ref{lem:bound} and \ref{lem:modulus} imply that
$C^{1,2}(\Lambda)$, $\X(\Lambda)$, $\M(\Lambda)$, $\M_{b}(\Lambda)$ are vector spaces; $C_{b}^{1,2}(\Lambda)$ is an algebra. 

\begin{Lem}\label{lem:phi}\noindent\\
Let $\Omega\subset Q^{\pi}_m$. If $\phi:\Lambda\longmapsto \R^{m\times m}$ is such that $\phi_{-}$ is left-continuous and locally bounded, then \begin{eqnarray*}(t,x_t)\in \Lambda \mapsto F(t,x_t):=\int_{0}^t\phi(s,x_{s-})d[x](s)\end{eqnarray*} is a continuous functional.
\end{Lem}

\begin{proof}
Since $t\longmapsto \phi(t,x_{t-})$ is left-continuous and locally bounded (Lemma~\ref{lem:cont}(i)) and that $t\longmapsto [x_i,x_j](t)$ is in $BV$, \cadlag with $\Delta [x_i,x_j]\equiv\Delta x_i\Delta x_j$ (Prop.~\ref{prop:qv}), it follows $F$ is a finite sum of Lebesgue-Stieltjes integrals and satisfies conditions Def.~\ref{def:pi}.1(a),(b)  and 2(a),(b). For the other conditions in Def.~\ref{def:pi}, it is suffice to assume $t_n\longrightarrow t$; $t_n\geq t'_n$ (i.e. the other criteria follow similar lines). Define\begin{eqnarray*}\phi_{n}(s):=\phi(t_0,x^{n}_{t_{0}-})\I_{\{0\}}(s)+\sum_{t_i\in \pi_n}\phi(t_i,x^{n}_{t_{i}-})\I_{(t_i,t_{i+1}]}(s),\end{eqnarray*}which is a $\R^{m\times m}$-valued left-continuous function on $\R_{+}$. By the local boundedness of $\phi_{-}$, we see that $\exists n_0\geq N(x)$; $(\phi_{n})_{n\geq n_0}$ is locally bounded on $\R_{+}$ and converges pointwise to $s\longmapsto\phi(s,x_{s-})$ on $\R_{+}$. By Cor.~\ref{cor:qs}(ii), we obtain \begin{alignat*}{2}
&F(t_n,x^n_{t_n})&=&\int_{0}^{t_n}\phi(s,x^n_{s-})d[x^n]\\
&&=&\sum_{\pi_n\ni t_i<t}\langle\phi(t_i,x^{n}_{t_{i}-}),(x(t_{i+1})-x(t_{i})(x(t_{i+1})-x(t_{i})'\rangle\longrightarrow F(t,x_t)\\
&&+&\sum_{\pi_n\ni t_i\in(t'_n,t_n]}\langle\phi(t_i,x^{n}_{t_{i}-}),(x(t_{i+1})-x(t_{i})(x(t_{i+1})-x(t_{i})'\rangle.\end{alignat*}Since $q_n\stackrel{\textsc{J$_1$}}{\longrightarrow}[x]$ and by \cite[s4.2]{CC}, the last term is bounded by
\begin{eqnarray*}
\text{const}|q_n(t_n)-q_n(t'_n)|\longrightarrow 0.\end{eqnarray*}
\end{proof}

 As we shall see in the following examples,  path-independent  functionals of class $\M$ are simply affine   functions, but in the path-dependent case this class includes many examples, in particular \follmer\  integrals. 
\begin{Eg}\label{eg:markov}\noindent\\
Let $\SS_m\subset\Omega$, $f\in C^{1,2}(\R_{+}\times\R^{m})$ and \begin{eqnarray*}F(t,x_t):=f(t,x(t)),\end{eqnarray*}then $F$ is of class $\M$ iff $f(t,u)=\alpha+\beta.u$ for some constants $\alpha\in\R, \beta\in\R^{m}$
\end{Eg}
\begin{proof}
For the if part: We can write $f(t,u)=\alpha+\beta\cdot u$ and hence \begin{eqnarray*}F(t,x_t)=\alpha+\beta x(t)\end{eqnarray*}on $\Lambda$ for some constants $\alpha\in\R$, $\beta\in\R^{m}$. By Example~\ref{eg:cont}(i) and computing the derivatives of $F$, we see that $F$ is of class $\M$. Conversely, from Def.~\ref{def:smooth} and Prop.~\ref{prop:causal}, we first obtain\begin{itemize}
    \item[(i)] $\partial_t f(t,x(t))=\D F(t,x_t)=0$,
    \item[(ii)] $\nabla^{2}f(t,x(t))=\nabla^{2}_{x}F(t,x_t)=0$,
\end{itemize}$\forall$ $t\geq 0$, $x\in \Omega$. Since $\SS_m\subset\Omega$, we have \begin{eqnarray*}
R:=\{(t,x(t))|t\in\R_{+}, x\in\Omega\}=\R_{+}\times\R^{m},
\end{eqnarray*}hence $\partial_t f\equiv\nabla^{2}f\equiv 0$ on $\R_{+}\times\R^{m}$. By the mean value theorem, we deduce that $\nabla f\equiv\beta$ on $R$, for some $\beta\in\R^{m}$.
\end{proof}

\begin{Rem}
The condition $\SS_m\subset\Omega$ may be weakened to simply requiring that $R\subset \R_{+}\times\R^{m}$ is convex. In this case, the converse statement holds on  $R$.
\end{Rem}

\begin{Eg}[Path-dependent examples]\label{eg:smooth}Let $\Omega\subset Q^{\pi}_m$, $\phi:\Lambda\longmapsto \R^{m\times m}$ such that $\phi_{-}$ is left-continuous and locally bounded, $f=(f_1,...,f_m)\in C^{2}(\R^{m})$. Then the functionals
\begin{itemize}
    \item[(i)]  $F(t,x_t):=\int_{0}^t\phi(s,x_{s-})d[x]$,
    \item[(ii)] $F(t,x_t):=\int_{0}^t (\nabla f\circ x) dx$,
    \item[(iii)]$F(t,x_t):=\sum_{i=1}^m\left(\int_{0}^t (x_i(t)-x_i(s))f_i(x_i(s)) dx_i(s)-\int_{0}^t(f_i\circ x_i) d[x_i]\right)$
\end{itemize} belong to $C_{b}^{1,2}(\Lambda)$ and (ii) and (iii) are of class $\M_b$.\end{Eg}

\begin{proof}
The functional in (iii) is well defined, since\begin{alignat}{2}F(t,x_t)=\sum_{i}\left(x_i(t)\int_{0}^t f_i\circ x_idx_i- \int_{0}^t x_if_i\circ x_i dx_i-\int_{0}^tf_i\circ x_i d[x_i]\right).\label{eq:smooth}
\end{alignat} The first two integrals in (\ref{eq:smooth}) are \follmer\  integrals, defined as a limit of Riemann sums along $\pi$, while the last one is a Lebesgue-Stieltjes integral. Continuity of $F$ in (i), (ii) and (iii) follows from Lemma~\ref{lem:phi} and Example~\ref{eg:cont}. Since $\D F\equiv 0$ in all cases, let us first compute $\nabla_{x}^{k}F$ for $k=1,2$ and demonstrate that $F$ possesses the required properties. In case of (i), we have \begin{alignat*}{2}\nabla_{x}F(t,x_t)=(\phi+\phi')(t,x_{t-})\Delta x(t), \quad\nabla_{x}^{2}F(t,x_t)=(\phi+\phi')(t,x_{t-}),\end{alignat*} which are left-continuous, locally bounded and $\nabla_{x}^{2}F$ is strictly causal, so by Prop.~\ref{prop:causal}, Lemma~\ref{lem:bound}(ii) and \eqref{lem:modulus}, $F$ is $C_{b}^{1,2}$. In case of (ii), we obtain \begin{alignat*}{2}\nabla_{x}F(t,x_t)=\nabla f(x(t-)),\end{alignat*}which is left-continuous, locally bounded and strictly causal, hence $F$ is of class $\M_b$. In case of (iii), we apply $\nabla_{x}$ to (\ref{eq:smooth}) and verify that
\begin{alignat}{2}\nabla_{x_i}F(t,x_t)&=\int_{0}^tf_i\circ x_idx_i-f_i(x_i(t-))\Delta x_i(t)\nonumber\\
&=\left(\int f_i\circ x_idx_i\right)(t-).\label{eq:smooth2}
\end{alignat} Applying $f(x):=\int_{0}^{x_i} f_i(\lambda)d\lambda$; $x\in \R^{m}$ to (ii) and by Prop.~\ref{prop:bound} and Lemma~\ref{lem:bound}(i), we see that each $\nabla_{x_i}F$ is left-continuous and locally bounded and so is $\nabla_{x}F$. Since $\nabla_{x}F$ is strictly causal, $F$ is of class $\M_b$.
\end{proof}

\section{Pathwise integration and change of variable formulas }\label{sec:ito}\noindent 

	We now discuss pathwise integration for causal functionals along paths in a generic domain. In contrast to rough integration theory \cite{frizshekhar} and the one form approach i.e \cite{HF}, \cite{CF} \& \cite{CP2019}, we define integrals as \emph{uncompensated} left Riemann sums, when such limits exist and form a continuous functional.  

	We then obtain change of variable formulas and an analogue of the classical Fundamental theorem of calculus for functionals of class $\M$. For paths that possess quadratic variation, we obtain a functional \follmer-\ito formula which extends \cite[Theorem 4]{CF}.
	
	In particular, we show that pathwise integral is of class $\M$ and that functionals of class $\M$ are \emph{primitives} i.e.  are representable as pathwise integrals, a  fact that facilitates the \emph{computation} of pathwise integrals, as in classical calculus. 

\begin{Lem}\label{lem:ftc}\noindent\\
Let $F$ be a left-continuous functional, differentiable in time, if $\D F$ is right-continuous and locally bounded, then \begin{eqnarray}
F(t,x_{s})-F(s,x_{s})=\int_{s}^{t}{\D F}(u,x_{u})du,
\end{eqnarray}for all $x\in \Omega$, $t\geq s\geq 0$.
\end{Lem}

\begin{proof}
Put $z:=x_s\in \Omega$, then $z_t=x_s$ for $t\geq s$ and $z_{t-}=x_s$ for $t>s$. Define $f(t):=F(t,x_s)$ for $t\geq s$, then $f(t)=F(t,z_t)$ on $[s,\infty)$ and $f(t)=F(t,z_{t-})$ on $(s,\infty)$. Since $F$ is differentiable in time, $f$ is right differentiable (hence right-continuous) on $[s,\infty)$ and the right derivative $f'(t)$ is $\D F(t,x_s)$ on $[s,\infty)$. Since $F$ is left-continuous, it follows from Lemma~\ref{lem:cont} that $f(t)=F(t,z_{t-})$ is left-continuous on $(s,\infty)$, hence we have first established that $f$ is continuous on $[s,\infty)$. Next, we observe that \begin{eqnarray*}f'(u)=\D F(u,x_s)=\D F(u,z_{u})\end{eqnarray*}on $[s,\infty)$. The right continuity of $\D F$ and Lemma~\ref{lem:cont} implies that $f'$ is right-continuous on $[s,\infty)$. Since $\D F$ is right-continuous and locally bounded, it follows from Lemma\ref{lem:bound}(ii) that \begin{eqnarray*}u\longrightarrow \D F(u,z_{u})\end{eqnarray*} is locally bounded. Hence, $f'$ is right-continuous and bounded on $[s,T]$, hence Riemann integrable. We can conclude using a stronger version \cite{FTC} of the Fundamental theorem of calculus.
\end{proof}

\begin{Lem}\label{lem:qs}
Let $\phi$ be a right-continuous and locally bounded on $\Lambda$, then
\begin{alignat*}{2}
\sum_{\pi_n\ni t_i \leq T}\int_{t_i}^{t_{i+1}}\phi(t,x^{n}_{t_{i}})dt\longrightarrow\int_0^{T}\phi(t,x_t)dt,
\end{alignat*}for all $x\in \Omega$, $T \geq 0$.
\end{Lem}

\begin{proof}
Define \begin{alignat*}{2}
{\phi}_{n}(t):=\sum_{\pi_n\ni t_i\leq T}\phi(t,x^{n}_{t_{i}})\I_{[t_i,t_{i+1})}(t)=\sum_{\pi_n\ni t_i\leq T}\phi(t,x^{n}_t)\I_{[t_i,t_{i+1})}(t).
\end{alignat*}By the local boundedness of $\phi$, we see that $\exists n_0\geq N(x)$; $({\phi}_n)_{n\geq n_0}$ is locally bounded on $[0,T]$. Since $\phi$ is right-continuous, it follows from Lemma~\ref{lem:cont} that $t\longmapsto{\phi}_n(t)$ is right-continuous (hence measurable) on $[0,T]$ and from Def.\ref{def:pi}.2(c) that ${\phi}_n$ converges to $t\longmapsto{\phi}(t,x_t)$ pointwise on $[0,T]$. and (i) follows from dominated convergence.
\end{proof}

\begin{Cor}\label{cor:DF}
Let $\phi$ be a right-continuous and locally bounded $\Lambda$, then \begin{eqnarray*}(t,x_t)\longmapsto\int_{0}^{t}\phi(s,x_s) ds\end{eqnarray*} is continuous.
\end{Cor}

\begin{proof}
The path $t\longmapsto \int_0^{t}\phi(s,x_s) ds$ is continuous. The rest follows from the local boundedness of $\phi$ and Lemma~\ref{lem:qs}.
\end{proof}

\begin{Def}[Pathwise integrability]\label{def:pathwise}\noindent\\
Let $\phi:\Lambda\longmapsto \R^{m}$ such that $\phi_{-}$ is left-continuous. For every $x\in\Omega$, define \begin{eqnarray}\label{eq:discrete}\II(t,x^n_{t}):=\sum_{\pi_n\ni t_i\leq t}\phi(t_i,x^n_{t_i-})\cdot(x(t_{i+1})-x(t_i)).\end{eqnarray} 
$\phi$ is said to be $\Lambda-$\emph{integrable} if
\begin{itemize}
    \item the limit $\II(t,x_t):=\lim_{n} \II(t,x^n_{t})$ exists for each $(t,x_t)\in \Lambda$, and
     \item the map $\II:\Lambda \mapsto \mathbb{R}$ is continuous.
\end{itemize}
\end{Def}
Note that the pathwise integral is defined as a limit of (left) Riemann sums, and not compensated Riemann sums as in rough path theory \cite{FrizHairer,frizshekhar}.
One case in which such Riemann sums are known to converge is for gradients of $C^2$ functions along paths of finite quadratic variation:

\begin{Eg} Let $\Omega= Q^\pi_m$.
Then by the results of \cite{HF}, for any $f\in C^2(\mathbb{R}^m),$  $\phi:\Lambda\longmapsto \R^{m}$ defined by $\phi(t,x)=\nabla_x f(t,x_t)$ is  $\Lambda$-integrable and $\II(t,x)$ is the \follmer\  integral \cite{RC}. Note that the continuity property of $\II$ is a consequence (and indeed, the main motivation) of the construction of the $\pi-$topology in Section \ref{sec:pi}.
\end{Eg}
\begin{Thm} \label{thm:exist}Let $\phi:\Lambda\longmapsto\R^{m}$ such that $\phi_{-}$ is left-continuous and $\II$ the integration map defined as in (\ref{eq:discrete}). If for every $x\in\Omega$, $T>0$   the sequence of step functions on $[0,T]$\begin{eqnarray*}
g_n(t):=\II(t,x^n_{t}),
\end{eqnarray*}is a Cauchy sequence in $(D[0,T],\d)$, then $\phi$ is $\Lambda$-integrable.
\end{Thm}

\begin{proof}
If $(g_n, n\geq 1)$ is a Cauchy sequence in $(D[0,T],\d)$, there exists a $G\in D$ such that $g_n\stackrel{\textsc{J}_1}{\longmapsto} G$. Hence $g_n(t)\mapsto G(t)$ for every continuity point of $G$ on $[0,T]$. Observe that \begin{eqnarray}\label{eq:delta}\Delta g_n(t)=\begin{cases}
    \phi(t_i,x^n_{t_i-})\cdot(x(t_{i+1})-x({t_i})), & \text{if $t=t_i\in\pi_n$}.\\
    0, & \text{otherwise}.
  \end{cases}
\end{eqnarray}If $\Delta G(t)>0$, there exists \cite[VI.2.1(a)]{JS} a sequence $t^*_n\rightarrow t$; $\Delta g_n(t^*_n)\rightarrow\Delta G(t)$. Using the fact that $\phi_{-}$ is left-continuous, $x$ is \cadlag and (\ref{eq:delta}), we see that 
\begin{eqnarray}
\lim_n \Delta g_n(t^*_n)=\phi(t,x_{t-})\cdot\Delta x(t)=\lim_n\phi(t'_n,x^n_{t'_n-})\cdot\Delta x^n(t'_n)=\lim_n \Delta g_n(t'_n),\label{eq:suff2}
\end{eqnarray}else we will contradict $\Delta G(t)>0$. Applying \cite[VI.2.1(b)]{JS}, we deduce that $(t^*_n)$ must coincide with $(t'_n)$ for all $n$ sufficiently large and by \cite[VI.2.1(b.3)]{JS}, we have established that \begin{eqnarray}
g_n(t)\longrightarrow G(t),\label{eq:suff3}
\end{eqnarray}hence we can define $\II(t,x_t):=G(t)$ on $[0,T]$.     Let $t''_n:=\min\{t_i>t'_n|t_i\in\pi_n\}$, $z:=x_{t-}\in\Omega$, it follows from (\ref{eq:discrete}), (\ref{eq:suff2}) and (\ref{eq:suff3}) that \begin{eqnarray*}
\II(t,x_{t-})=\lim_n \II(t,z^n_t)=\lim_n \left(\II(t,x^n_t)-\phi(t'_n,x^n_{t'_n-})\cdot(x(t''_n)-x(t-))\right)=G(t-),
\end{eqnarray*}hence $t\longmapsto \II(t,x_t)$ is \cadlag and its jump at time $t$ is $\II(t,x_t)-\II(t,x_{t-})$. If $t_n\longrightarrow t$, the limits of $g_n(t_n)$ and $g_n(t_{n}-)$ are readily determined according to (\ref{eq:suff2}) and \cite[VI.2.1(b)]{JS}. The continuity criteria in Def.~\ref{def:pi} are thus satisfied.
\end{proof}

\begin{Prop}\label{prop:L}Let $\phi$ be $\Lambda$-integrable. Then $\D\II=0$ and $\nabla_{x}\II=\phi_{-}$ on $\Lambda$.
\end{Prop}
  
\begin{proof}
Let $(t,x)\in\Lambda$ and $z:=x+e\I_{[t,\infty)}\in\Lambda$. Then \begin{alignat*}{2}
\II(t,z_t)-\II(t,x_t)&=\lim_{n}\left(\II(t,z^n_t)-\II(t,x^n_t)\right)\\
&=\lim_{n}\phi(t'_n,z^n_{t'_n-})\cdot e\\
&=\lim_{n}\phi(t'_n,x^n_{t'_n-})\cdot e=\phi(t,x_{t-})\cdot e,
\end{alignat*}by the continuity of $\II$ and left-continuity of $\phi_{-}$. 
\end{proof}

\begin{Thm}[Change of variable formula for class $\X$ functionals]\label{thm:formulaX}\noindent\\
Let  $F\in\X(\Lambda)$. Then for any $(T,x_T)\in \Lambda$,  the limit
\begin{alignat}{2}\label{eq:Pathwise}
\int_{0}^{T}\nabla_{x}{F}(t,x_{t-})dx:=\lim_{n\to\infty}\sum_{\pi_n\ni t_i\leq T}\nabla_{x}{F}(t_i,x^n_{t_i-})\cdot(x(t_{i+1})-x(t_i))\end{alignat} exists and
\begin{alignat*}{2}
F(T,x_T)=F(0,x_0)+\int_{0}^{T}{\D F}(t,x_t)dt+\int_{0}^{T}\nabla_{x}F(t,x_{t-})dx.
\end{alignat*}
\end{Thm}
\begin{proof}
Appendix \S~\ref{sec:append}.
\end{proof}

\begin{Rem}\label{rem:Pathwise}
By Prop.~\ref{prop:L}, we see that all pathwise integrals are functionals of class $\M$, hence by Thm.~\ref{thm:formulaX}, we can write \begin{eqnarray}\II(t,x_t)=\int_0^t\phi dx. \label{eq:I}\end{eqnarray}
\end{Rem}As we shall see, the converse is also true, all integrals that may be defined by (\ref{eq:Pathwise}) are pathwise integrals in the sense of Def.~\ref{def:pathwise}:
\begin{Cor}[Decomposition for class $\X$]\label{cor:decom}Let $F\in\X(\Lambda)$. Then $M:\Lambda\to \mathbb{R}$ defined by \begin{alignat*}{2}
M(t,x_t):=F(t,x_t)-F(0.x_0)-\int_{0}^{t}{\D F}(s,x_s)ds
\end{alignat*}is of class $\M$ and $\nabla_{x}M=\nabla_{x}F$. In particular, $M$ may be represented as a pathwise integral: there exists a $\Lambda-$integrable functional $\phi:\Lambda\to \mathbb{R}^m$ such that $M=\II$:
$$ \forall (t,x)\in \Lambda, \qquad M(t,x)= \int_0^t \phi.dx$$
\end{Cor}

\begin{proof}
By differentiating $M$, we obtain $\D M=0$ and $\nabla_{x}M=\nabla_{x}F$. Continuity of $M$ follows from Corollary~\ref{cor:DF} and Theorem~\ref{thm:formulaX}, hence  by (\ref{eq:Pathwise}), $M$ satisfies Def.~\ref{def:pathwise}.
\end{proof}

In fact, all functionals of class $\M$ have an integral representation.
 We obtain as a corollary a Fundamental theorem of calculus for functionals:
\begin{Cor}\label{cor:ftc}\noindent
\begin{itemize}
    \item[(i)] Let $\phi$ be $\Lambda$-integrable. Then the map  $\II: (t,x_t)\in \Lambda \mapsto \int_0^t \phi .dx$ is continuous, differentiable and \begin{eqnarray*}\nabla_x \II=\phi_{-}.\end{eqnarray*}
    \item[(ii)] Let $\phi:\Lambda \to\mathbb{R}$. If $F\in  \M (\Lambda)$ such that $\nabla_{x}F=\phi_{-}$, then $\phi$ is $\Lambda$-integrable and \begin{eqnarray*}\int_0^{t}\phi dx=F(t,x_t)-F(0,x_0).\end{eqnarray*}
\end{itemize}
\end{Cor}
\begin{proof}
(i) is due to Prop.~\ref{prop:L} and Rem.~\ref{rem:Pathwise}. (ii) is due to (\ref{eq:Pathwise}) and Cor.~\ref{cor:decom}.
\end{proof}

\begin{Eg}\label{eg:ito}
Let $\Omega\subset Q^{\pi}_m$, $f_i\in C^1(\R)$, then
\begin{alignat}{2}\label{eq:integrable2}
\int_{0}^{T}\left(\int f_1\circ x_1dx_1, \right.&\left.\ldots,\int f_m\circ x_mdx_m\right)'dx\nonumber\\
=&\sum_i\left(\int_{0}^{T} (x_i(T)-x_i)f_i\circ x_i dx_i-\int_{0}^{T}f_i\circ x_i d[x_i]\right),\end{alignat}by an application of Cor.~\ref{cor:ftc}(ii) to the RHS of (\ref{eq:integrable2}), Example.~\ref{eg:smooth}(iii) and (\ref{eq:smooth2}).
\end{Eg}



  An important consequence of Theorem \ref{thm:formulaX} is to show that  class $\M$ functionals satisfy a pathwise analogue of the  \emph{martingale} property.
The concept of  {martingale} was originally introduced to model the outcome of a  \emph{fair game} \cite{ville} across a set of outcomes.
         The following result, which does not make use of any probabilistic notion, shows that a class $\M$ functional represents the outcome of such a 'fair game', where the underlying set of outcomes is a generic subset of paths:
\begin{Thm}[Fair game]\label{thm:fair}\noindent\\
Let $ M\in \M(\Lambda)$. If there exists $T>0$ such that \begin{eqnarray*}\forall x\in \Omega,\qquad M(T,x_{T})-M(0,x_0)\geq 0 \end{eqnarray*}then $$\forall x\in \Omega,\quad  M(T,x_{T})=M(0,x_0).$$
 \end{Thm}
This result suggests that class $\M$ functionals may be considered pathwise analogues of     {martingales}.

\begin{proof}Since $\D M$ vanishes, by Lemma~\ref{lem:ftc} we obtain\begin{eqnarray}\label{eq:non-negativ}
M(t,x_{t})=M(t,x_{t})+\int_{t}^{T}\D M(s,x_t)ds= M(T,x_{t})\geq 0
\end{eqnarray}for all $t\leq T$, where the last inequality is due to $x_t\in\Omega$. Suppose there exists $z\in\Omega$; $M(T,z_T)>0$. By Thm.~\ref{thm:formulaX} and the continuity of $M$, it follows \begin{eqnarray}M(T,z^{n}_T)=\sum_{\pi_n\ni t_i\leq T}\nabla_{x}{M}(t_i,z^n_{t_i-})(z(t_{i+1})-z(t_i))>0\label{eq:martingale}\end{eqnarray} for all $n$ sufficiently large. Define $t^{*}_n:=\min\{t_i\in\pi_n | M(t_i,z^{n}_{t_i})>0\}$, then $t^{*}_n\leq T$. By (\ref{eq:non-negativ}), (\ref{eq:martingale}), the left-continuity of $M$ and the fact that $z^{n}\in\Omega$, we obtain  \begin{eqnarray*}
M(t^{*}_n,z^{n}_{t^{*}_n})>M(t^{*}_n,z^{n}_{t^{*}_n-})=0,
\end{eqnarray*}hence $M(t^{*}_n,z^{n}_{t^{*}_n})=\nabla_{x}{M}(t^{*}_n,z^n_{t^{*}_n-})\Delta z(t^{*}_n)>0$. Def.~\ref{def:closed}(ii) implies that there exists $\epsilon>0$ such that \begin{eqnarray*} z^{*}:=z^{n}_{t^{*}_n-}-\epsilon\Delta z(t^{*}_n)\I_{[t^{*}_n,\infty)}\in\Omega,\end{eqnarray*}hence $
M(t^{*}_n,z^{*}_{t^{*}_n})=\nabla_{x}{M}(t^{*}_n,z^n_{t^{*}_n-})(-\epsilon\Delta z(t^{*}_n))<0$, which  contradicts \eqref{eq:non-negativ}.
\end{proof}

   The following change of variable formula for causal functionals extends \cite[Theorem 4]{CF} to $C^{1,2}(\Lambda)$, removing the condition linking the partition sequence $\pi$ with the jump times of a path:
\begin{Thm}[Change of variable formula for $C^{1,2}$ functionals]\label{thm:formula}\noindent\\
Let $x\in\Omega\cap Q^{\pi}_m$. For any $F\in C^{1,2}(\Lambda)$   the following \follmer-\ito formula holds:\begin{alignat}{2}
F(T,x_T)&=F(0,x_0)+\int_{0}^{T}{\D F}(t,x_t)dt+\int_{0}^{T}\nabla_{x}F(t,x_{t-})dx\label{eq:form}\\
&+\frac{1}{2}\int_{0}^{T}\nabla^2_{x}F(t,x_{t-})d[x]^{c}+\sum_{t\leq T}\left(\Delta{F}(t,x_t)-\nabla_{x}{F}(t,x_{t-})\cdot\Delta{x}(t)\right),\nonumber
\end{alignat} where the series is absolute convergent and the pointwise limit
\begin{alignat}{2}
\int_{0}^{T}\nabla_{x}{F}(t,x_{t-})dx:=\lim_{n\to \infty}\sum_{\pi_n\ni t_i\leq T}\nabla_{x}{F}(t_i,x^n_{t_i-})\cdot(x(t_{i+1})-x(t_i))\label{eq:path} \end{alignat}exists.\end{Thm}

\begin{proof}
See Appendix \S~\ref{sec:append}.
\end{proof}
An important consequence of Theorem \ref{thm:formula} is the continuity of the \follmer\ integral in the $\pi-$topology:
\begin{Prop}\label{prop:ftc}
Let $\Omega\subset Q^{\pi}_{m}$ and $F\in C^{1,2}(\Lambda)$. Then\begin{eqnarray*}
J: \Lambda &\longmapsto &\mathbb{R}\nonumber\\
(t,x) &\longmapsto& J(t,x_t):=\int_{0}^{t}\nabla_{x}{F}(s,x_s)d x
\end{eqnarray*}is continuous. In particular, $\nabla_{x}F$ is integrable and $J$ is a pathwise integral in the sense of Def.~\ref{def:pathwise}.
\end{Prop}    

\begin{proof}
We apply the functional change of variable formula (Thm.~\ref{thm:formula}) to $F$. Rearranging the terms in (\ref{eq:form})  we observe that $t\longmapsto J(t,x_t)$ is \cadlag whose jump at time $t$ is $J(t,x_t)-J(t,x_{t-})$. It remains to show that $J$ satisfies the continuity criteria Def.~\ref{def:pi}.1(c),(d) and 2(c),(d). It is suffice to assume $t_n\longrightarrow t$; $t_n\geq t'_n$ (i.e. the other criteria follow similarly). By (\ref{eq:path}) and that $x$ is right-continuous, we first obtain\begin{eqnarray}
J(t_n,x^n_{t_n})&=&\int_{0}^{t_n}\nabla_{x}F(t,x^n_{t-})dx^n\nonumber\\
&=&\sum_{\pi_n\ni t_i<t}\nabla_{x}F(t_i,x^n_{t_{i}-})\cdot(x(t_{i+1})-x(t_{i}))\longrightarrow J(t,x_t)\nonumber\\
&+&\sum_{\pi_n\ni t_i\in(t'_n,t_n]}\nabla_{x}F(t_i,x^n_{t_{i}-})\cdot(x(t_{i+1})-x(t_{i})).\label{eq:rest} \end{eqnarray}We have to show that the rest term (\ref{eq:rest}) vanishes as $n\uparrow\infty$. Applying (\ref{eq:form}) to the path $x^n$ and by the local boundedness of $\D F$, we have \begin{alignat*}{2}
\left |\sum_{\pi_n\ni t_i\in(t'_n,t_n]}\nabla_{x}F(t_i,x^n_{t_{i}-})\cdot\Delta x^n(t_i)\right | &\leq |F(t_n,x^n_{t_n})-F(t'_n,x^n_{t'_n})|\\
&+\text{const} |t_n-t'_n|\\
&+\left |\sum_{\pi_n\ni t_i\in(t'_n,t_n]}\Delta F(t_i,x^n_{t_i})-\nabla_{x}F(t_i,x^n_{t_i-})\cdot\Delta x^n(t_i)\right |.
\end{alignat*}Since $t_n\geq t'_n$; $t_n, t'_n\longrightarrow t$ and by the right continuity of $F$ the first two terms vanish. Since $(\nabla^2_{x}F)_{-}$ is locally bounded and $\nabla^2_{x}F$ admits a modulus, applying a second order Taylor expansion to the third term, we obtain\begin{alignat*}{2}
\left |\sum_{\pi_n\ni t_i\in(t'_n,t_n]}\Delta F(t_i,x^n_{t_i})-\nabla_{x}F(t_i,x^n_{t_i-})\cdot\Delta x^n(t_i)\right |\leq \text{const}|q_n(t_n)-q_n(t'_n)|\longrightarrow 0,
\end{alignat*}by the fact that $q_n\stackrel{\textsc{J$_1$}}{\longrightarrow}[x]$ and \cite[\S4.2]{CC}.
\end{proof}

\section{Application to paths with finite quadratic variation}\label{sec:ftc}\noindent 

    We now examine in more detail the case of paths of finite quadratic variation   and apply the results developed in \S.\ref{sec:ito} to the case $\Omega\subset Q^{\pi}_m$. As we have already shown, integration and differentiation are inverse operations (Cor.~\ref{cor:ftc}). Using functionals of class $\M$, we show that these operations may be viewed as \emph{isomorphisms} between certain spaces. We also obtain a pathwise identity related to It\^o's isometry (Theorem \ref{thm:isometry}).
    
    The key objects here are functionals of class $\M$, which are \emph{primitives} (e.g.~\ref{eg:ito}) and may be understood as the  pathwise analogue of  {martingales} (Thm.~\ref{thm:fair}). In addition, we shall show that class $\M$ are canonical \emph{solutions} to  path dependent heat equations. 
    Let us introduce the following vector spaces of integrands:\begin{alignat*}{2}
L(\Lambda):=\{\nabla_{x}F|F\in C^{1,2}(\Lambda)\},&\qquad
L_{b}(\Lambda):=\{\nabla_{x}F|F\in C_{b}^{1,2}(\Lambda)\},\nonumber\\
\L(\Lambda):=\{\nabla_{x}F|F\in \M(\Lambda)\},&\qquad
\L_{b}(\Lambda):=\{\nabla_{x}F|F\in \M_{b}(\Lambda)\}.
\end{alignat*}

    By Prop.~\ref{prop:ftc}, the integral operator  
 \begin{alignat*}{2}\s: \phi\in L(\Lambda)\longmapsto \II \in \R^{\Lambda},
 \end{alignat*}where $\II$ is given by \eqref{eq:I}, is a well-defined linear operator. 
 
\begin{Eg}[Path-dependent 1-form]\label{eg:left_int}\noindent\\
Let $f_i\in C^{1}(\R)$, $i=1,\ldots,m$ then\begin{eqnarray*}
\phi(t,x_t):=\left(\left(\int f_1\circ x_1 dx_1\right)(t-),\ldots,\left(\int f_m\circ x_m dx_m\right)(t-)\right)'
\end{eqnarray*} defines  an element of $\L_{b}(\Lambda)$.
\end{Eg}

\begin{proof}
See Example~\ref{eg:smooth}(\ref{eq:smooth2}).
\end{proof}

\begin{Lem}\label{lem:image}\noindent
\begin{itemize}
    \item[(i)] If $\phi\in L(\Lambda)$ then $\s\phi\in\M(\Lambda)$ and $\nabla_x(\s\phi)=\phi_{-}$.
    \item[(ii)] If $\phi\in L_{b}(\Lambda)$ then $\s\phi\in\M_{b}(\Lambda)$ and $\nabla_x(\s\phi)=\phi_{-}$.
    \item[(iii)] If $\phi\in \L(\Lambda)$ then $\s\phi\in\M(\Lambda)$ and $\nabla_x(\s\phi)=\phi$.
    \item[(iv)] If $\phi\in \L_{b}(\Lambda)$ then $\s\phi\in\M_b(\Lambda)$ and $\nabla_x(\s\phi)=\phi$.
\end{itemize}
\end{Lem}

\begin{proof}
It is due to Prop.~\ref{prop:ftc} and Cor.\ref{cor:ftc}(i). 
\end{proof}

\begin{Cor}\label{cor:biject}
Define \begin{eqnarray*}\M_{0}(\Lambda):=\{F\in \M_{b}(\Lambda)|F(0,x_0)\equiv 0\},\end{eqnarray*} then the integral operator \begin{eqnarray*}
\s:\L_{b}(\Lambda)\longmapsto \M_{0}(\Lambda)
\end{eqnarray*}is an isomorphism and the inverse of $\s$ is the differential operator $\nabla_{x}$.
\end{Cor}

\begin{proof}
Injectivity follows from Lemma~\ref{lem:image}(iv). Surjectivity is due to Cor.~\ref{cor:ftc}(ii).
\end{proof}
    We now obtain a pathwise identity of \ito\footnote{First appeared in \cite[Lem. 2]{ito}.}, in the spirit\footnote{Probabilistic equals to deterministic counterpart, up to a martingale term.} of the pathwise Burkholder-Davis-Gundy inequality \cite{PS}, and give an application.
For $\phi,\psi\in \L_{b}(\Lambda)$ define $\{\phi,\psi\}\in  \L_{b}(\Lambda)$ by
\begin{eqnarray*} \{\phi,\psi\}: \Lambda& \mapsto & \mathbb{R}^d \nonumber\\
(t,x)&\to  &\left(\psi \int_0^{.}\phi.dx+\phi\int_0^{.} \psi. dx\right)(t,x_{t-}).\end{eqnarray*}

\begin{Thm}\label{thm:isometry} For all $\phi, \psi \in \L_b(\Lambda)$, $\{\phi,\psi\}\in \L_{b}(\Lambda)$ and
\begin{eqnarray*}
\left(\int\phi dx\right)\left(\int\psi dx\right)=\int\phi\psi' d[x]+
\int \{\phi,\psi\} dx.
\end{eqnarray*} 
\end{Thm}

\begin{proof} Recall that
  $C^{1,2}_{b}(\Lambda)$ is an algebra. Let $\phi, \psi \in \L_{b}(\Lambda)$, put $F:=\int\phi dx, G:=\int\psi dx$, then $F,G\in \M_b(\Lambda)$ by Lemma~\ref{lem:image}(iv). Since $\M_b(\Lambda)\subset C^{1,2}_b(\Lambda)$, it follows $FG\in C^{1,2}_b$. Apply the change of variable formula (Thm.~\ref{thm:formula}) to $FG$; using Lemma \ref{lem:image}(ii), the proof is complete.
\end{proof}

\begin{Cor}[Isometry]\label{cor:isometry}\noindent\\
Let $\E\subset \L_b(\Lambda)$ be a subspace such that
$$ \forall \phi, \psi \in \E, \quad \{\phi,\psi\}\in \E$$
and denote $\bf{I}(\E)$ the image of $\E$ under $\s$. If $\mathbb{E}$ is any positive element of the algebraic dual $C^{*}(\Lambda)$ such that $\bf{I}(\E) \subset \text{ker} (\mathbb{E})$, then \begin{eqnarray*}
 \left\langle\int\phi dx, \int\psi dx\right\rangle_{\bf{I}(\E)}:=\mathbb{E}\left(\int\phi dx \int\psi dx\right)=\mathbb{E}\left(\int\phi\psi' d[x]\right)=:\left\langle\phi,\psi\right\rangle_{\E}
\end{eqnarray*}holds for all $\phi, \psi\in\E$. 

    In particular, the bracket $\left\langle.,.\right\rangle_{\E}$ induces a semi-norm on $\E$. Denoting $\tilde{\E}$ the quotient space induced by the semi-norm,  the integral operator\begin{alignat*}{2}
\tilde{\s}:\tilde{\E}&\longmapsto \bf{I}(\tilde{\E})\\
\tilde{\phi}&\longmapsto \tilde{\s}\tilde{\phi}:=\s\phi 
\end{alignat*}is an isometric isomorphism between the pre-Hilbert spaces $\tilde{\E}$ and $\bf{I}(\tilde{\E})$. The inverse of $\tilde{\s}$ is the differential operator\begin{alignat*}{2}
\tilde{\nabla_{x}}:\bf{I}(\tilde{\E})&\longmapsto \tilde{\E} \\
\tilde{F}&\longmapsto \tilde{\nabla_{x}}\tilde{F}:=\nabla_{x}F 
\end{alignat*}
\end{Cor}

\begin{proof}
The result is  a consequence of Cor.~\ref{cor:biject} and Thm.~\ref{thm:isometry}.
\end{proof}

 
    We conclude with a discussion on the relation between class $\M(\Lambda)$  and harmonic functionals, defined as solutions to a class of path-dependent heat equations \cite[Ch. 8]{RC}. Let $\Sigma: \Lambda\to S_m^+$ be a right-continuous function on $\Lambda$ taking values in  positive-definite symmetric $m\times m$ matrices and \begin{eqnarray*}
\Omega_{\Sigma}:=\{x\in \Omega|\frac{d[x]}{dt}=\Sigma \}\subset \Omega
\end{eqnarray*}
the set of paths with absolutely continuous quadratic variation with Lebesgue density $\Sigma$.
 \begin{Def}
 $F\in C^{1,2}(\Lambda)$ is called $\Sigma$-harmonic if it satisfies\begin{eqnarray}
\forall x\in \Omega_{\Sigma}, \quad \forall t\geq 0, \qquad 
\D F(t,x_t)+\frac{1}{2}\langle\nabla^{2}_{x}F(t,x_{t}),\Sigma(t,x_t)\rangle=0.\label{eq:pde}
\end{eqnarray}
 \end{Def}  

If $F$ is $\Sigma$-harmonic, then the change of variable formula  (Theorem \ref{thm:formula}) 
gives \begin{eqnarray}
F(t,x_t)=F(0,x_0)+\int_{0}^{t}\nabla_{x}F(s,x_{s-})dx\label{eq:harmonic}
\end{eqnarray}for all $t\geq 0$ and $x\in \Omega_{\Sigma}$. Equality in (\ref{eq:harmonic}) then holds on $\Omega_{\Sigma}$. 
Every functional of class $\M$ satisfies \eqref{eq:pde}, hence is $\Sigma$-harmonic for all $\Sigma$. 
\begin{Thm}[Representation of $\Sigma$-harmonic functionals]
\label{thm:harmonic}
If $F$ is $\Sigma$-harmonic, then there exists a class $\M$ functional $M$ such that \begin{eqnarray*}
M|_{\Omega_{\Sigma}}\equiv F.
\end{eqnarray*}In particular,  $M$ is uniquely determined  by (\ref{eq:harmonic}) on $\Omega_{\Sigma}$.
\end{Thm}
\begin{proof}
Let $F\in C^{1,2}(\Lambda)$ be $\Sigma$-harmonic. We can define a  functional $M:\Lambda\to\mathbb{R}$ by \begin{eqnarray}
M(t,x):=F(0,x_0)+\int_{0}^{t}\nabla_{x}F(s,x_{s-})dx.\label{eq:harmonic2}
\end{eqnarray}By Lemma~\ref{lem:image}(i), we see that $M\in\M(\Lambda)$ and $\nabla_{x}M=(\nabla_{x}F)_{-}$. By (\ref{eq:harmonic}) and (\ref{eq:harmonic2}), the proof is complete. 
\end{proof}

\section{Technical proofs}\label{sec:append}

\begin{proof}[\bf{Proof of Prop.~\ref{prop:invalid}}]
 For  $\alpha\in\R_{+}$, define $w_{\alpha}(t):=\I_{[\alpha,\infty)}(t)\in D=:\Omega$, where $D$ denotes the Skorokhod space. We assign to the collection $(w_{\alpha})_{\alpha\in\R_{+}}$, a normalized Lebesgue measure \begin{eqnarray*}\mathbb{P}(\{w_{\alpha}|\alpha\in A\}):=\sum_{n\geq 1}\frac{\lambda(A\cap[0,n])}{2^{n+1}},\end{eqnarray*}then $\mathbb{P}(\{w_{\alpha}|\alpha\in\R_{+}\}))=1$ and $X_t(w):=w(t)$ is a finite variation process (i.e. a semi-martingale) under $\mathbb{P}$. 
 Now let $\pi=(\pi_n)_{n\geq 1}$ be any sequence of time partitions and denote \begin{eqnarray*}Q^{\pi}_{0}:=\{x\in Q^{\pi}|J(x)\subset\liminf_{n}\pi_n\}.\label{eq.QV0}\end{eqnarray*} 
 Since $\liminf_{n}\pi_n$ is countable, it follows that $\mathbb{P}(\{w_{\alpha}|\alpha\in\liminf_{n}\pi_n\})=0$ and therefore $\mathbb{P}(\{\omega\in\Omega| X_{\cdot}(\omega)\in Q^{\pi}_{0}\})=0$.
\end{proof}

\begin{proof}[\bf{Proof of Theorems~\ref{thm:formulaX} and \ref{thm:formula}}]
By the right continuity of $F$ (Def.~\ref{def:pi}.2(d)), we have\begin{eqnarray}\label{eq:f-1}
F(T,x_T)-F(0,x_0)=\lim_n\sum_{\pi_n\ni t_i\leq T}F(t_{i+1},x^n_{t_{i+1}-})-F(t_{i},x^n_{t_{i}-}),
\end{eqnarray}where for all $n$ sufficiently large, we can decompose each increments
\begin{alignat*}{2}
&F(t_{i+1},x^n_{t_{i+1}-})-F(t_{i},x^n_{t_{i}-})\\
=&F(t_{i+1},x^n_{t_{i+1}-})-F(t_{i},x^n_{t_{i+1}-})+F(t_{i},x^n_{t_{i+1}-})-F(t_{i},x^n_{t_{i}-})\\
=&\underbrace{\left(F(t_{i+1},x^n_{t_{i}})-F(t_{i},x^n_{t_{i}})\right)}_{\text{time}}+\underbrace{\left(F(t_{i},x^n_{t_{i}})-F(t_{i},x^n_{t_{i}-})\right)}_{\text{space}}\end{alignat*}into the sum of a time ('horizontal') and a space ('vertical') increment. 

	Since $F$ is left-continuous and differentiable in time, $\D F$ is right-continuous and locally bounded, by Lemma~ \ref{lem:ftc} each time increment may be expressed as\begin{alignat*}{2}
F(t_{i+1},x^n_{t_{i}})-F(t_{i},x^n_{t_{i}})=\int_{t_i}^{t_{i+1}}{\D F}(t,x^n_{t_i})dt.\end{alignat*}By Lemma \ref{lem:qs}, we obtain\begin{eqnarray*}
\lim_n\sum_{\pi_n\ni t_i\leq T}F(t_{i+1},x^n_{t_{i}})-F(t_{i},x^n_{t_{i}})=\int_{0}^{T}{\D F}(t,x_t)dt,\end{eqnarray*}which in light of (\ref{eq:f-1}), implies that the sum of  space increments converges to\begin{eqnarray}
\lim_n\sum_{\pi_n\ni t_i\leq T}\underbrace{F(t_{i},x^n_{t_{i}})-F(t_{i},x^n_{t_{i}-})}_{\Delta F(t_{i},x^n_{t_{i}}) }=F(T,x_T)-F(0,x_0)-\int_{0}^{T}{\D F}(t,x_t)dt.\label{eq:f1}
\end{eqnarray}

    If $F\in\X(\Lambda)$ then $\nabla_{x}F$ is strictly causal and by Prop.~\ref{prop:causal}, $\nabla_{x}^{2}F$ is vanishing everywhere. Thus, by a second order Taylor expansion, the remainder term vanishes, so \begin{alignat*}{2}F(t_{i},x^n_{t_{i}})-F(t_{i},x^n_{t_{i}-}) =\nabla_{x}F(t_i,x^n_{t_i-})\cdot\left(x(t_{i+1})-x(t_{i})\right)\end{alignat*}and Thm.~\ref{thm:formulaX} follows. If $F\in C^{1,2}(\Lambda)$ then, by Taylor's Theorem, each space increment admits the following second order expansion\begin{alignat}{2}\Delta F(t_{i},x^n_{t_{i}})&=F\left(t_{i},x^n_{t_{i}-}+\Delta x^n(t_{i})\I_{[t_i,\infty)}\right)-F(t_{i},x^n_{t_{i}-})\nonumber\\
&=\nabla_{x}F(t_i,x^n_{t_i-})\cdot\Delta x^n(t_{i})+\frac{1}{2}\langle\nabla^2_{x}F(t_i,x^n_{t_i-}),\Delta x^n(t_{i})\Delta x^n(t_{i})'\rangle\nonumber\\
&+R^n_{t_i},\label{eq:taylor} \end{alignat}where $\Delta x^n(t_{i})=\left(x(t_{i+1})-x(t_{i})\right)$ and\begin{eqnarray*}
R^n_{t_i}=\frac{1}{2}\langle\nabla^2_{x}F(t_i,x^n_{t_i-}+\alpha^n_i\Delta x^n(t_{i})\I_{[t_i,\infty)})-\nabla^2_{x}F(t_i,x^n_{t_i-}),\Delta x^n(t_{i})\Delta x^n(t_{i})'\rangle\end{eqnarray*}where $\alpha^n_i\in(0,1)$. Since $x\in\Omega_2\subset Q^{\pi}_m$, by Cor.~\ref{cor:qs} and Rem.~\ref{rem:qv}\begin{alignat}{2}
&\lim_n\sum_{\pi_n\ni t_i \leq T}\langle\nabla^2_{x}F(t_i,x^{n}_{t_{i}-}),\Delta x^n(t_{i})\Delta x^n(t_{i})'\rangle=\int_{0}^{T}\nabla^2_{x}F(t,x_{t-})d[x]\nonumber\\
&=\int_{0}^{T}\nabla^2_{x}F(t,x_{t-})d[x]^c+\sum_{t\leq T}\langle\nabla^2_{x}F(t,x_{t-}),\Delta x(t)\Delta x(t)'\rangle.\label{eq:qs}
\end{alignat} 

	Let $\delta>0$, $r:=\sup_{t\in [0,T]}|\Delta x(t)|$, $r_{\delta}:=\delta+\sup_{t\in [0,T+\delta]}|\Delta x(t)|$. Using a result on \cadlag functions \cite[Lemma 8]{CF}, we see that $|\Delta x^n(t_{i})|\leq r_{\delta}$ for  $n$ sufficiently large. By Rem.~\ref{rem:closed}, we see that $\alpha^n_i\Delta x^n(t_{i})\in \U_{t_i-}(x^n)\cap\overline{B}_{r_{\delta}}(0)$. Since $\nabla^2_{x}F$ admits a modulus of vertical continuity, it follows from Def.~\ref{def:space modulus} that there exists a modulus of continuity $\omega$ such that\begin{eqnarray*}|R^n_{t_i}|\leq\frac{1}{2}\omega(r_{\delta})|\Delta x^n(t_{i})\Delta x^n(t_{i})'|\end{eqnarray*}for  $n$ sufficiently large, hence by an application of Cor.~\ref{cor:qs}(i), we obtain\begin{eqnarray*}\limsup_n\sum_{\pi_n\ni t_i\leq T}|R^n_{t_i}|\leq\frac{1}{2}\omega(r_{\delta})\leq\omega(r_{\delta})tr\left([x](T)\right).\end{eqnarray*}Send $\delta\downarrow 0$, and by the right continuity of $x$, we have established that\begin{eqnarray}\limsup_n\sum_{\pi_n\ni t_i\leq T}|R^n_{t_i}|\leq\frac{1}{2}\omega(r+)tr\left([x](T)\right).\label{eq:rb}\end{eqnarray} 
	Let $0<\epsilon<r$, define the following finite sets on $[0,T]$\begin{alignat*}{2}J(\epsilon)&:=\{t\leq T| |\Delta x(t)|>\epsilon\},\\J_n(\epsilon)&:=\{\pi_n\ni t_i\leq T| \exists t\in(t_i,t_{i+1}], |\Delta x(t)|>\epsilon\}.\end{alignat*}We can decompose\begin{alignat}{2}
\sum_{\pi_n\ni t_i\leq T}R^n_{t_i}=\sum_{t_i\in J_n(\epsilon)}R^n_{t_i}+\sum_{t_i\in \left(J_n(\epsilon)\right)^{c}}R^n_{t_i}.\label{eq:rd}\end{alignat}into two partial sums. By (\ref{eq:taylor}), the right continuity (resp. left-continuity) of $F$ (resp. $(\nabla_{x}F)_{-}$,$(\nabla^2_{x}F)_{-}$) and that $x$ is \cadlag we obtain\begin{alignat}{2}
\sum_{t_i\in J_n(\epsilon)}\left(R^n_{t_i}\right)^{\pm}\stackrel{n}{\longrightarrow}&\sum_{t\in J(\epsilon)}\left(\vphantom{\frac{1}{2}}\Delta F(t,x_t)-\nabla_{x}F(t,x_{t-})\cdot\Delta x(t)\right.\nonumber\\
&\left.-\frac{1}{2}\langle\nabla^2_{x}F(t,x_{t-}),\Delta x(t)\Delta x(t)'\rangle\vphantom{\frac{1}{2}}\right)^{\pm}\nonumber\\
&\leq\frac{1}{2}\omega(r+)tr\left([x](T)\right),\label{eq:rd1}
\end{alignat}where the inequality follows from (\ref{eq:rb}) and (\ref{eq:rd}). Observe that $J(\epsilon)\uparrow J(0)$ as $\epsilon\downarrow 0$, by monotone convergence, we obtain\begin{alignat}{2}
\lim_n\sum_{t_i\in J_n(\epsilon)}\left(R^n_{t_i}\right)^{\pm}\stackrel{\epsilon}{\longrightarrow}&\sum_{t\leq T}\left(\vphantom{\frac{1}{2}}\Delta F(t,x_t)-\nabla_{x} F(t,x_{t-})\cdot\Delta x(t)\right.\nonumber\\
&\left.-\frac{1}{2}\langle\nabla^2_{x}F(t,x_{t-}),\Delta x(t)\Delta x(t)'\rangle\vphantom{\frac{1}{2}}\right)^{\pm}\nonumber\\
&\leq\frac{1}{2}\omega(r+)tr\left([x](T)\right).\label{eq:rd2}
\end{alignat}

	On the other hand, since $w$ is monotonic, by (\ref{eq:rb}) and (\ref{eq:rd}), it follows that\begin{eqnarray}
\left|\limsup_n\sum_{t_i\in \left(J_n(\epsilon)\right)^{c}}R^n_{t_i}-\liminf_n\sum_{t_i\in \left(J_n(\epsilon)\right)^{c}}R^n_{t_i}\right|\leq\omega(\epsilon)tr\left([x](T)\right),\label{eq:re}
\end{eqnarray}and by (\ref{eq:f1})-(\ref{eq:qs}), (\ref{eq:rd}),(\ref{eq:rd1}) and (\ref{eq:re}), so is\begin{eqnarray*}
\left|\limsup_n\sum_{\pi_n\ni t_i\leq T}\nabla_{x}{F}^n_{t_i}\cdot\Delta x^n(t_{i})-\liminf_n\sum_{\pi_n\ni t_i\leq T}\nabla_{x}{F}^n_{t_i}\cdot\Delta x^n(t_{i})\right|\\
\leq\omega(\epsilon)tr\left([x](T)\right),
\end{eqnarray*}where we have denoted $\nabla_{x}{F}^n_{t_i}:=\nabla_{x}{F}(t_i,x^n_{t_i-})$. Send $\epsilon\downarrow 0$, we obtain \begin{eqnarray}
\int_{0}^{T}\nabla_{x}{F}(t,x_{t-})dx:=\lim_{n}\sum_{\pi_n\ni t_i\leq T}\nabla_{x}{F}(t_i,x^n_{t_i-})\cdot(x(t_{i+1})-x(t_i)).\label{eq:int}\end{eqnarray}
	
	Upon a second look at (\ref{eq:f1})-(\ref{eq:qs}), (\ref{eq:rd}),(\ref{eq:rd1}) and in light of (\ref{eq:int}), we immediately see that \begin{eqnarray*}\lim_n\sum_{t_i\in \left(J_n(\epsilon)\right)^{c}}R^n_{t_i}=:o(\epsilon)\end{eqnarray*} also exists and by (\ref{eq:rb}), $|o(\epsilon)|\leq\frac{1}{2}\omega(\epsilon)tr\left([x](T)\right)\stackrel{\epsilon}{\longrightarrow}0$ which, combined with (\ref{eq:rd}) and (\ref{eq:rd2}) implies\begin{alignat}{2}
\lim_n\sum_{\pi_n\ni t_i\leq T}R^n_{t_i}&=\sum_{t\leq T}\left(\vphantom{\frac{1}{2}}\Delta F(t,x_t)-\nabla_{x}F(t,x_{t-})\cdot\Delta x(t)\right.\nonumber\\
&\left.-\frac{1}{2}\langle\nabla^2_{x}F(t,x_{t-}),\Delta x(t)\Delta x(t)'\rangle\vphantom{\frac{1}{2}}\right).\label{eq:rd3}
\end{alignat}

	In view of (\ref{eq:f1})-(\ref{eq:qs}), (\ref{eq:int}) and (\ref{eq:rd3}), it remains to show that \begin{alignat}{2}
	&\sum_{t\leq T}\left(\Delta F(t,x_t)-\nabla_{x}F(t,x_{t-})\Delta x(t)-\frac{1}{2}\langle\nabla^2_{x}F(t,x_{t-}),\Delta x(t)\Delta x(t)'\rangle\right)\nonumber\\
=&\sum_{t\leq T}\left(\vphantom{\frac{1}{2}}\Delta F(t,x_t)-\nabla_{x}F(t,x_{t-})\Delta x(t)\vphantom{\frac{1}{2}}\right)-\frac{1}{2}\sum_{t\leq T}\langle\nabla^2_{x}F(t,x_{t-}),\Delta x(t)\Delta x(t)'\rangle,\label{eq:abs}
\end{alignat}and the absolute convergence of the series.
	Since $(\nabla^2_{x}F)_{-}$ is left-continuous and locally bounded, we see from Lemma~\ref{lem:bound}(ii) that the map $t\longmapsto \nabla^2_{x}F(t,x_{t-})$ is also bounded on $[0,T]$, hence by (\ref{eq:qv})\begin{alignat*}{2}
\frac{1}{2}\sum_{t\leq T}|\nabla^2_{x}F(t,x_{t-})||\Delta x(t)\Delta x(t)'|&\leq\text{const}\sum_{i}\left(\sum_{t\leq T}(\Delta x_i(t))^2\right)\\
&\leq\text{const}\cdot tr\left([x](T)\right),
\end{alignat*} which, combined with (\ref{eq:rd2}) implies \eqref{eq:abs} and  the absolute convergence of the series, hence Theorem~\ref{thm:formula} is proven.
\end{proof}

\end{document}